\documentclass[a4paper,10pt]{article}

  \usepackage{paralist}
  \usepackage{afterpage}
\usepackage{xcolor}
  \usepackage{graphics} %% add this and next lines if pictures should be in esp format
  \usepackage{epsfig} %For pictures: screened artwork should be set up with an 85 or 100 line screen
 \usepackage[colorlinks=true]{hyperref}

\hypersetup{urlcolor=blue, citecolor=red}
 
\usepackage[utf8]{inputenc}
\usepackage{amssymb, amsmath, amsfonts, amsthm}
\usepackage{commath}
\usepackage{graphicx}
\usepackage{xcolor}
\definecolor{darkblue}{rgb}{0,0,0.5} 
\usepackage{transparent}
\usepackage[numbers,sort]{natbib}
\usepackage{subfig}
\usepackage{bigints}

\usepackage[disable]{todonotes}

\newtheorem{theorem}{Theorem}[section]
\newtheorem{lemma}[theorem]{Lemma}
\newtheorem{proposition}[theorem]{Proposition}
\newtheorem{corollary}[theorem]{Corollary}

\theoremstyle{definition}
\newtheorem{definition}[theorem]{Definition}
\newtheorem{remark}[theorem]{Remark}

\author{Elena Celledoni\footnote{Email: \href{mailto:elena.celledoni@math.ntnu.no}{elena.celledoni@math.ntnu.no}}, Markus Eslitzbichler\footnote{Email: \href{mailto:markus.eslitzbichler@math.ntnu.no}{markus.eslitzbichler@math.ntnu.no}} and Alexander Schmeding\footnote{Email: \href{mailto:alexander.schmeding@math.ntnu.no}{alexander.schmeding@math.ntnu.no}}}
\title{Shape Analysis on Lie Groups with Applications in Computer Animation}

\DeclareMathOperator{\dexp}{dexp}
\date{}
\begin{document}

\newcommand{\Diffeom}{\mathrm{Diff}^+(I)}
\newcommand{\LieG}{G}
\newcommand{\LieA}{\mathfrak{g}}
\newcommand{\SOT}{SO(3)}
\newcommand{\sot}{\mathfrak{so}(3)}
\newcommand{\SOTCurves}{C^\infty(M, \SOT)}
\newcommand{\sotCurves}{C^\infty(M, \sot)}
\newcommand{\Evol}{\text{Evol}^r_{\SOT}}

\newcommand{\RightTrans}{R}
\newcommand{\LeftTrans}{L}
\newcommand{\SRVT}{\mathcal{R}}
\newcommand{\SRVTInv}{\mathcal{R}^{-1}}

\newcommand{\LQ}{r(q)}
\newcommand{\adj}{\dagger}

\newcommand{\SSpace}{\mathcal{S}}
\newcommand{\PSpace}{\mathcal{P}}
\newcommand{\Energy}{\mathcal{E}}

\maketitle

\bigskip
\begin{abstract}

Shape analysis methods have in the past few years become very popular, both for theoretical exploration as well as from an application point of view.
Originally developed for planar curves, these methods have been expanded to higher dimensional curves, surfaces, activities, character motions and many other objects.

In this paper, we develop a framework for shape analysis of curves in Lie groups for problems of computer animations.
In particular, we will use these methods to find cyclic approximations of non-cyclic character animations and interpolate between existing animations to generate new ones.
\end{abstract}
\medskip

metric, Infinite-dimensional manifolds, Lie groups,  Computer animation.

\medskip


\tableofcontents

\section{Introduction}

Motions of virtual characters in movies and interactive applications are usually represented using a skeletal animation approach where the data consists of curves tracking the positions of the bones throughout the motion.
These curves can be processed by mathematical methods to produce new motions \cite{pejsa_state_2010}.
It has previously been shown that shape analysis methods can be successfully applied to solve problems in computer animation by considering entire animations as curves and shapes \cite{eslitzbichler_modelling_2014, bauer_landmark-guided_2015}.
In practice, the data consists of curves in $SO(3)^d$, where $d$ is the number of bones in the skeleton.
However, in the earlier approaches, curves have been represented using Euler angles, neglecting the underlying Lie group structure.
We here report on the results we obtained by appropriately including the underlying geometric structure in the mathematical models and their numerical discretizations.
The intrinsic geometric formulation is robust and works very well in problems of motion blending and curve closing, where earlier the same performance could be only obtained  by using ad hoc strategies, e.g. keeping track of carefully chosen feature points along the curves.

In section~\ref{background}, we will briefly introduce shape analysis and motivate how techniques from shape analysis on Lie groups can be applied to computer animation by treating character animations as points in an infinite dimensional manifold.
This manifold is in fact an infinite dimensional Lie group where we are interested in computing distances and geodesics.

In Section \ref{Sec:ShapeAnalysis}, we discuss some of the main tools for shape analysis on Lie groups, which are later applied in this paper. 
An approach to curves evolving on Riemannian manifolds was earlier presented in \cite{su_rate-invariant_2014}. 
In our work, we exploit the additional structure provided by the Lie group setting.
To perform animation blending (i.e., interpolation between existing animations to create a new one), we are interested in simple and computationally efficient approaches to compute geodesics between two given shapes and to this end we define a metric on shape space. We show that this metric is associated to the geodesic distance of the pullback of an $L^2$-inner product (cf.\ Theorem \ref{thm: pbmetric} and Theorem \ref{thm: globaldist}).
The Lie group formalism allows us in Section \ref{Sec:ClosedCurves} to efficiently solve the curve closing problem (which we use to approximate an existing non-periodic animation with a periodic one) using a gradient flow approach.
Finally, in Section \ref{Sec:NumericalResults}, we present numerical results both for problems in animation blending and animation closing.

We demonstrate that the proposed techniques exhibit better qualitative performance compared to previous work \cite{eslitzbichler_modelling_2014, bauer_landmark-guided_2015}.
In animation blending the new approach allows interpolating between a wider range of motions.
In animation closing the proposed Lie group formulation is naturally intrinsic, and allows avoiding undesired artefacts due to coordinate singularities.
While our specific applications use the special orthogonal group $\SOT$ as the underlying Lie group, the techniques developed in this paper are not restricted to this setting.

\section{Background}
\label{background}

\subsection{Shape analysis}

Many problems in object and activity recognition can be formulated in terms of similarities of \emph{shapes} \cite{bauer_constructing_2014, srivastava_statistical_2005,srivastava_shape_2011,eslitzbichler_modelling_2014, klassen_path-straightening_2005, mio_shape_2007,younes_spaces_2012, younes_computable_1998, fuchs_shape_2009, sharon_2d-shape_2006, cotter_reparameterisation_2012}.
By shapes we typically mean unparametrized curves in a vector space or on a manifold, although similar methods have been developed for surfaces \cite{bauer_new_2011, kurtek_novel_2010, kurtek_elastic_2012, bauer_sobolev_2011}.

In recent years, a number of methods based on differential geometry have been developed to tackle such problems.
Of particular importance is the question of how to model and work with unparametrized curves.

A popular approach is to define shapes as equivalence classes of certain mappings, where the equivalence relation is induced by reparametrization.
Given two curves $c_0, c_1: I \rightarrow M$ with $I = [a,b] \subset \mathbb{R}$ and $M$ a vector space or a manifold, we define equivalence classes $[c_0], [c_1]$ via the equivalence relation:
\[
c_0 \sim c_1 \iff \exists \, \varphi: c_0 = c_1 \circ \varphi,
\] where $\varphi$ is a smooth, strictly increasing bijection on $I$.
We denote by $\PSpace$ the space of parametrized curves containing $c_0$ and $c_1$.
Typical choices of $\PSpace$ include absolutely continuous functions, immersions, embeddings or piecewise linear functions \cite{bauer_overview_2014, lahiri_precise_2015}.
The equivalence classes, or shapes, can then be collected in the corresponding \emph{shape space}:
\[
    \SSpace := \PSpace /_\sim.
\]

Shape analysis then concerns itself with the study of the spaces $\SSpace$ and $\PSpace$ from both a theoretical and a practical perspective.

Many applications require for example a distance function in $\SSpace$ to measure similarities between shapes, which can be used to perform statistical analysis such as clustering and object recognition \cite{srivastava_statistical_2005}.
Distance functions on shape space $\SSpace$ are typically obtained from a distance function $d_\PSpace$ on the underlying space $\PSpace$ of parametrized curves as follows:
\[
    d_\SSpace([c_0], [c_1]) := \inf_{\varphi} d_\PSpace (c_0, c_1 \circ \varphi),
\] where $\varphi$ ranges over all possible curve reparametrizations.
The computation of $d_\SSpace$ amounts to the solution of an optimization problem with appropriate numerical techniques \cite{sebastian_aligning_2003, srivastava_shape_2011}.

In recent years, a distance function on $\PSpace$ based on the so-called \emph{elastic metric} has become popular.
The elastic metric is a first-order Sobolev-type metric, which is easy to compute and has desirable theoretical properties \cite{michor_overview_2007, bauer_sobolev_2011, srivastava_shape_2011, bauer_overview_2014}.

In this paper, we consider curves and shapes in Lie groups to solve problems in computer animation.
In that context the use of Lie groups arises naturally by using Euclidean transformations and rotations to describe motions of virtual characters.

\subsection{Skeletal animation}
In computer graphics, used and seen in movies, tv-series and video games, but also in educational and scientific software, virtual characters are most commonly represented as surfaces in $\mathbb{R}^3$.

Motions of such characters are usually represented using a \emph{skeletal animation} approach.
The underlying skeleton consists of bones connected by joints.
The vertices of the surface mesh are attached to bones, i.e., their positions are specified in a coordinate system that is aligned with a bone.
Then, when the skeleton is animated by specifying the positions of all the joints as a function of time, the vertices move accordingly.
See Figure \ref{Fig:3DCharacter} for an example virtual character.

\begin{figure}
\center
\includegraphics[trim = 9cm 1cm 0 2cm, clip=true, scale=0.2]{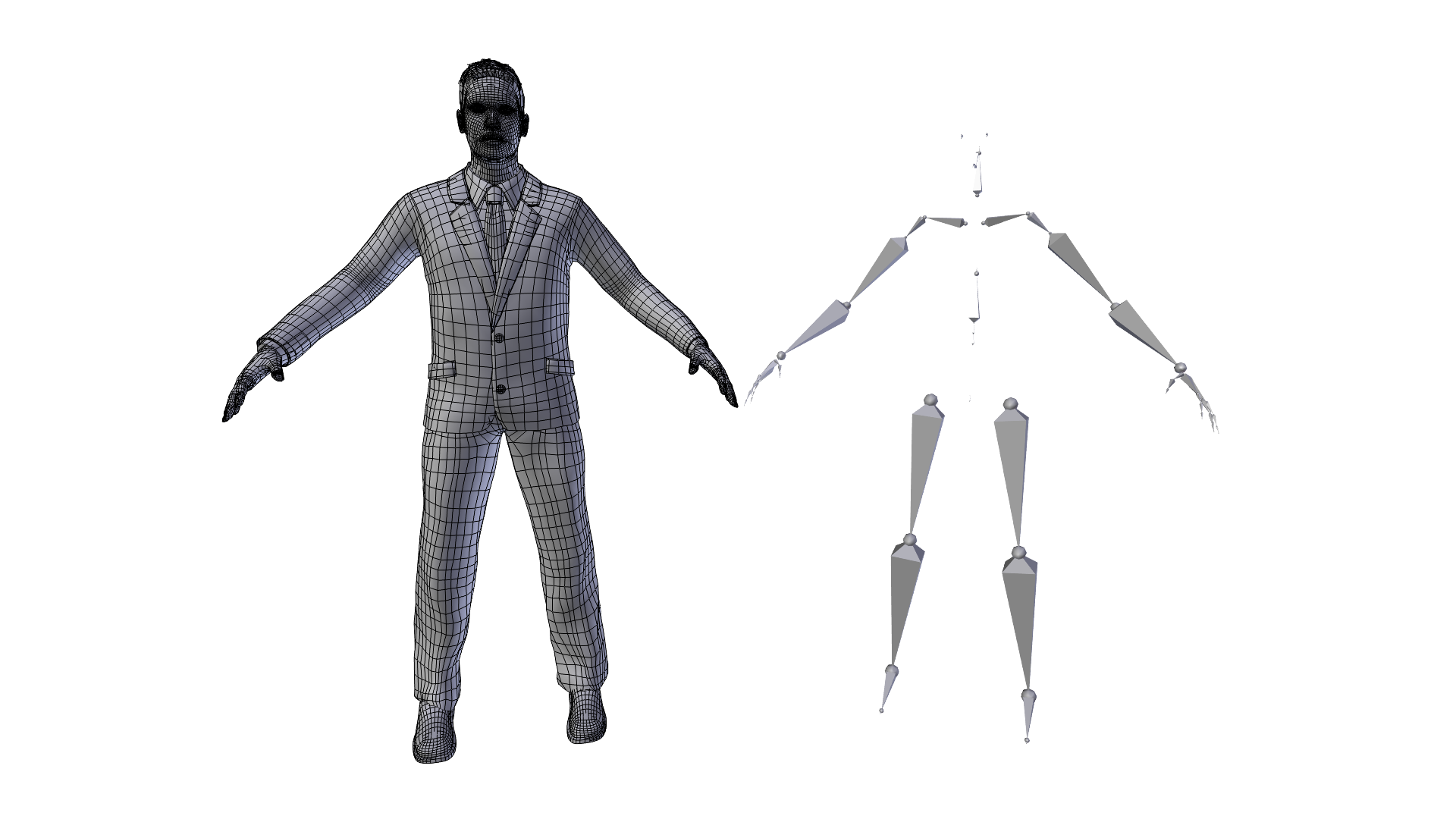} 
%\vspace{-1cm}
\includegraphics[trim = 9cm 0.5cm 0 2cm, clip=true, scale=0.2]{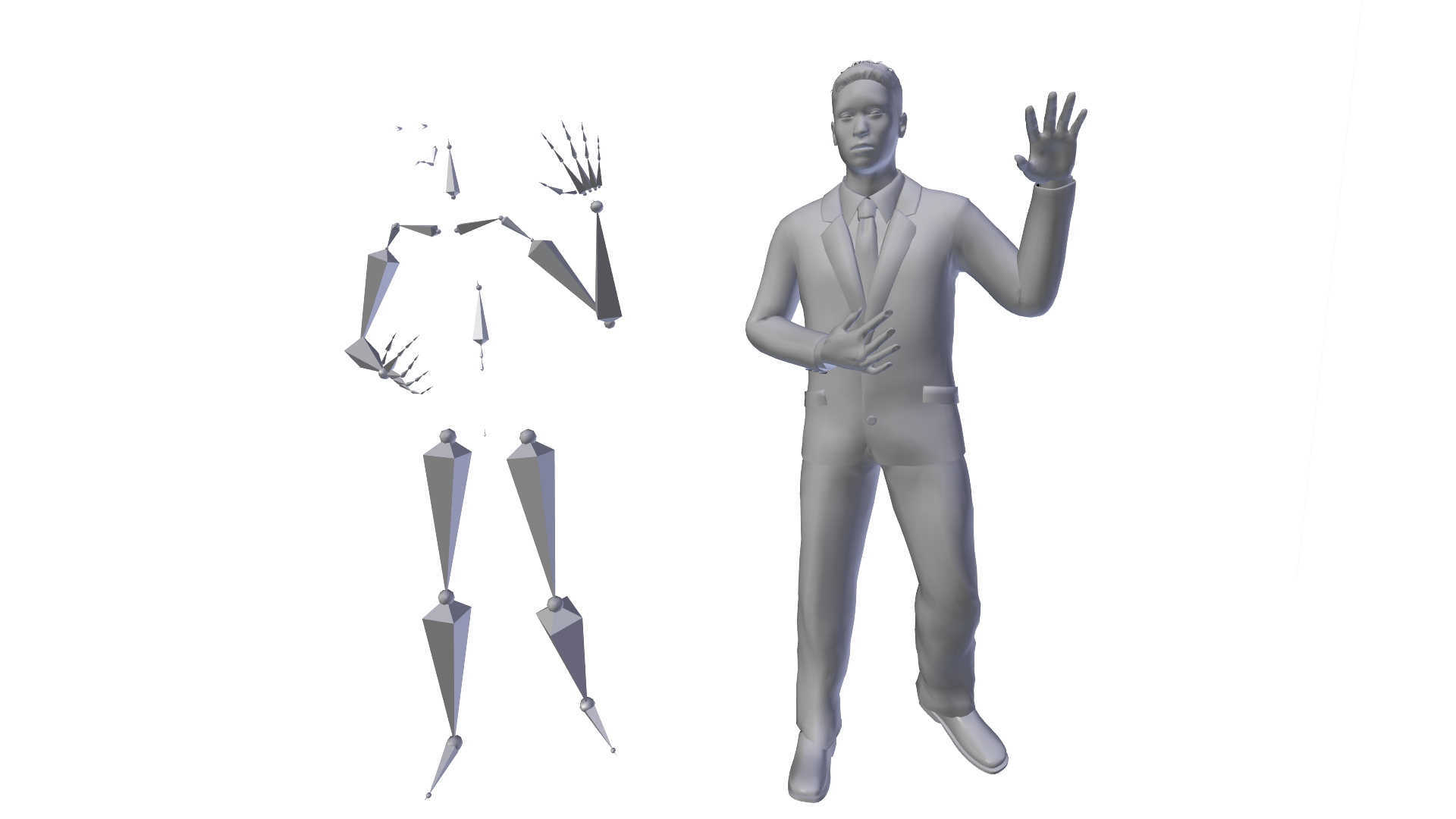} 
%\vspace{-1cm}
\caption{\label{Fig:3DCharacter} Surface mesh depicting a human character with underlying skeleton.
Only bones with a degree of freedom are shown, i.e., for example the hips are not visible in the skeleton and the legs appear at an offset from the character's center.}
\end{figure}

A skeleton is a rooted tree, where each node (bone) represents a Euclidean coordinate system.
The edge (joint) between a node and its parent represents a Euclidean transformation indicating how the two nodes' coordinate systems are positioned and oriented relative to each other.
By following this kinematic chain of relative transformations from a node to the root, all local coordinate systems can be brought into a single global coordinate system.

A character's pose is specified by assigning values to all degrees of freedom in the skeleton (i.e., all the joints).
Each such configuration is an element of \emph{joint space} $\mathcal{J} := SE(3)^d$.
An animation is then a function from some time interval $[a,b]$ into joint space, specifying a character pose for every point in time.\\
For human characters, bones have fixed lengths and joint space therefore consists only of rotations between bones (i.e., elements in $\SOT$, the special orthogonal group) instead of more general Euclidean transformations ($SE(3)$).
\medskip

A typical approach of generating such animations for use in videos or games is \emph{motion capturing}, where an actor's motions are recorded from multiple points of view and the underlying skeletal motion is extracted.
This is often preferred over alternative means of generating animations, such as manual construction or inverse kinematics.

Motion capturing however suffers from limitations inherent in its process, namely that the data is static and limited in range.
If we want an animated character to run at different speeds, we need corresponding recordings.
If we want a virtual character to keep walking forward for an indeterminate amount of time (e.g., in a video game under a player's control), we need to find a way of adapting a finite walking animation so that we can repeat it without visible discontinuities.

Much work has been done to procedurally manipulate motion data to tackle these and many more problems in computer graphics.
We refer to, among many others, \cite{kovar_motion_2002,kovar_flexible_2003,kovar_automated_2004, gonzalez_castro_cyclic_2010, pejsa_state_2010, shoemake_animating_1985} for an overview and some specific approaches in this field.
\\
To take into account the geometry inherent to skeletal animation, in the following sections, we propose methods for motion blending and animation closing on Lie groups.

\section{Shape analysis on Lie Groups}\label{Sec:ShapeAnalysis}

In this section we will develop a framework for shape analysis for curves on Lie groups. 
Albeit our main application is shape analysis for curves on $\SOT$, the theoretical framework exhibited here is fairly general.
The methods discussed in this and the next section can also be applied to certain classes of infinite-dimensional Lie groups.

In the following, $G$ will refer to a Lie group with Lie algebra $(\mathfrak{g}, [\cdot,\cdot])$, identity $e$ and multiplication $\, \cdot: G \times G \rightarrow G$. 
Here $G$ might even be an infinite-dimensional Hilbert Lie group, i.e.\ a Lie group in the sense of \cite{kriegl_regular_1997} (or \cite{MR2261066}) modelled on a Hilbert space.\footnote{Assuming that the Lie group is modelled on a Hilbert space will assure that our results carry over without any change. However, some of the results extend even to Banach Lie groups, cf.\ Remark \ref{rem: frame:infty}.}
We will denote left and right translations by $\LeftTrans_g$ and $\RightTrans_g$ respectively, i.e., $\LeftTrans_{g_1} ( g_2 ) = g_1 \cdot g_2$ and $\RightTrans_{g_1} ( g_2 ) = g_2 \cdot g_1$.
We further denote by $\Gamma$ the \emph{evolution operator}, which is defined as
\begin{align*}
&\Gamma: C^\infty(I, \mathfrak{g}) \rightarrow \lbrace c \in C^\infty(I, G): c(0) = e \rbrace =: C^\infty_* (I, \LieG)\\
&\Gamma (q) (t) := c(t) \qquad \text{where} \qquad \od{}{t}c(t)=\RightTrans_{c(t)*}(q(t)),\quad c(0)=e,
\end{align*} where $I \subset \mathbb{R}$ is an interval with $0 \in I$ and $\RightTrans_{g*} = T_e R_g$ is the tangent map of the right translation at the identity $e$.
Without loss of generality, we let $I=[0,1]$.

The inverse of the evolution operator is the so called \emph{right logarithmic derivative} 
  \begin{align*}
   &\delta^r \colon C^\infty_* (I, \LieG) \rightarrow C^\infty (I,\LieA), \\
   &\delta^r c  := (R_{c}^{-1})_* (\dot{c}).
  \end{align*}

\subsection{Manifolds of smooth mappings}
In this section we recall the construction of the manifold structure on spaces of smooth mappings with values in a Lie group.
Moreover, we review some basic facts on these manifolds which will be used throughout the article.
Following \cite[Chapter 42]{kriegl97tcs}, we construct the manifold structure on $C^\infty (I,G)$ using a local addition on $G$.
A local addition allows us to choose local parametrisations on a manifold in a smooth way.  

\begin{definition}[{Local Addition, cf.\ \cite[42.4]{kriegl97tcs}}]
 A local addition on a manifold $M$ is a smooth map $\Sigma \colon TM \supseteq \Omega \rightarrow M$ defined on an open neighbourhood $\Omega$ of the zero-section such that 
 \begin{compactitem}
  \item $\Sigma (0_x) = x$ for all $x \in M$, where $0_x \in T_x M$ is the zero element,
  \item $(\pi_{TM}, \Sigma) \colon \Omega \rightarrow M \times M, \quad \xi_x \mapsto (\pi_{TM} (\xi_x) , \Sigma (\xi_x)) = (x , \Sigma (\xi_x))$ induces a diffeomorphism onto a neighbourhood of the diagonal in $M\times M$.
  Here $\pi_{TM}$ is the tangent bundle projection.
 \end{compactitem}
\end{definition}
 
Since $G$ is a Lie group we can construct a local addition as follows.
Choose a chart $\psi \colon \LieG \supseteq V_\psi \rightarrow W_\psi \subseteq \LieA$ around the identity in $\LieG$.
Denoting elements in the tangent space over $g \in \LieG$ by $\xi_g$, we obtain a local addition via 
\begin{displaymath}
 \Sigma \colon T\LieG \supseteq \Omega := \bigcup_{g \in G} R_{g*} (W_\psi) \rightarrow \LieG ,\quad \xi_g \mapsto g \cdot \psi^{-1} (R_{g*}^{-1} (\xi_g)).
\end{displaymath}

Now define $C^\infty_\alpha (I,T\LieG) := \{f \in C^\infty (I,T\LieG) \mid \pi_{T\LieG} \circ f = \alpha\}$ for $\alpha \in C^\infty(I,G)$. 
This is a vector space with the pointwise operations (in the fibres $T_{\alpha (t)} G$ for each $t$).

In the following we endow spaces of smooth functions $C^\infty (I,M)$ to a (possibly infinite-dimensional) manifold $M$ with the compact open $C^\infty$-topology.
This topology allows to control functions and their partial derivatives on any compact subset of $I$. We refer to \cite[Section I.5]{MR2261066} for more information on this topology.
Then $U_\alpha := \{h\in C^\infty (I,\LieG) \mid \forall t \in I,\ (\alpha(t),h(t)) \in \text{Im} (\pi_{TG} ,\Sigma)\}$ is open in $C^\infty (I,G)$ and one can prove that the assignment 
\begin{displaymath}
 \varphi_\alpha \colon C^\infty (I,G) \supseteq U_\alpha \rightarrow C^\infty_\alpha (I,T\LieG),\quad f \mapsto (\pi_{TG}, \Sigma)^{-1} \circ (\alpha,f)
\end{displaymath}
is a manifold chart for $C^\infty (I,\LieG)$.
By \cite[Theorem 7.8]{MR3351079} (and the remarks before the cited theorem) these charts turn $C^\infty (I,G)$ into a Fr\'{e}chet manifold (i.e.\ a manifold modelled on a locally convex space which is complete and metrizable).

At this point we leave the realm of Banach manifolds, whence the standard definition for smooth maps (i.e.\ viewing the derivative as a continuous map to a space of continuous operators) breaks down.
One way to define smooth maps beyond the Banach setting is the so called Bastiani calculus (see \cite{bastiani64}): A map between Fr\'{e}chet spaces is smooth if all iterated directional derivatives exist and are continuous in a natural sense (see e.g.\ \cite[I.2]{MR2261066} for more details). There are other (in general inequivalent) ways to define smooth mappings on infinite-dimensional spaces, such as the so called convenient calculus (see \cite{kriegl97tcs}). Fortunately, these choices yield the same smooth maps on Fr\'{e}chet manifolds. As we exclusively work with Fr\'{e}chet manifolds we can thus disregard the differences and freely use results formulated in both calculi.\footnote{Note that we have already done this, as \cite{MR3351079} (Bastiani calculus) just generalizes \cite[Section 42]{kriegl97tcs} (convenient calculus). Further, both calculi can handle smooth maps on the non-open domain $I = [a,b]$, see e.g.\ \cite[Definition 7.2]{MR3351079} and \cite[Chapter 24]{kriegl97tcs}.} 

Furthermore, we quote from \cite[41.10, Theorem 42.13]{kriegl97tcs} and \cite[Theorem 7.9]{MR3351079} the following results which we will use later on. 

\begin{proposition}\label{prop: LGP:prop}\mbox{}
\begin{compactenum}
   \item Pointwise multiplication and inversion induce a Fr\'{e}chet Lie group structure on $C^\infty (I,G)$. 
   Its Lie algebra is $C^\infty (I,\LieA)$ and we let $e_{\LieG} \colon I \rightarrow \LieG$ be its identity.
   \item The set $\mathrm{Imm}(I,\LieG)$ is open in $C^\infty (I,\LieG)$ and for an open subset $U \subseteq \LieG$, the set $C^\infty (I,U) := \{ f \in C^\infty (I,\LieG) \mid f(I) \subseteq U\}$ is open.
   \item By \cite[1.9]{glockner_regularity_2012}, $C_{*}^{\infty}(I, \LieG):=\lbrace c \in C^\infty(I, \LieG): c(0) = e \rbrace$ is a closed Lie subgroup with Lie algebra $\{f \in C^\infty (I,\LieA) : f(0)=0\}$.\footnote{Here we have used that $\varphi_{e_G} = \psi_{*} \colon C^\infty (I,V_\psi) \rightarrow C^\infty (I,W_\psi), f \mapsto \psi\circ f$, where $C^\infty (I,W_\psi) := \{f \in C^\infty (I,\LieA) \mid f(I)\subseteq W_\psi\}$ is open in $C^\infty (I,\LieA)$. The Lie group structure from 1. coincides with the one defined in \cite{glockner_regularity_2012}.}    
   \item Using the bijection $(\cdot)^\wedge \colon C^\infty (\mathbb{R}, C^\infty (I,G)) \rightarrow C^\infty (\mathbb{R} \times I, G), \quad \gamma^\wedge (\varepsilon,t) := \gamma(\varepsilon) (t)$ one identifies the tangent bundle $TC^\infty (I,G)$ with $C^\infty (I,TG)$. 
   Fibrewise this isomorphism is given by  
    \begin{displaymath}
     T_\alpha C^\infty (I,\LieG) \ni \left(\left.\frac{\mathrm{d}}{\mathrm{d}\varepsilon}\right|_{\varepsilon =0}\gamma \right) \longmapsto \left(\left.\frac{\partial}{\partial \varepsilon}\right|_{\varepsilon = 0} \gamma^\wedge (\varepsilon, \cdot) \right) \in C^\infty_\alpha (I,TG)
    \end{displaymath}
    where $\gamma \colon \mathbb{R} \rightarrow C^\infty (I,\LieG)$ is smooth, with $\gamma (0) = \alpha$.  
 \end{compactenum}
\end{proposition}

\subsection{Shape space and distance functions}
We want to measure distances between two shapes, i.e., unparametrized curves on $\LieG$.
To do this, one represents unparametrized curves as equivalence classes of parametrized curves under certain reparametrizations.

We will model parametrized curves in $\LieG$ as immersions (i.e., smooth functions with a nonvanishing first derivative) and denote the space of these curves by $\mathcal{P} := \text{Imm}(I, \LieG)$.
Then one can define the \emph{shape space} $\mathcal{S}$ as the quotient space
\[
    \mathcal{S} := \PSpace/\Diffeom,
\] where $\Diffeom$ denotes the group of orientation preserving diffeomorphisms on $I$, acting on $\PSpace$ from the right \cite{bauer_overview_2014}.
In the context of parametrized curves, $\Diffeom$ can be thought of as the group of all possible orientation preserving parametrizations of a curve.

Our goal is to find a distance function $d_\SSpace$ on $\SSpace$.
To this end, we consider an appropriate distance function on $\PSpace$.
The distance of two parametrized curves will be measured as the infimum of the length of piecewise smooth curves connecting these curves (i.e.\ the geodesic distance of a Riemannian metric).  
Note that in general, the distance will only be a pseudometric\footnote{A pseudometric $d \colon \PSpace \times \PSpace \rightarrow \mathbb{R} \cup \{\infty\}$ satisfies all axioms of a metric but might fail to distinguish different points, i.e. in general $d(x,y)=0$ will not imply $x=y$.}, see Remark \ref{rem: pseudometric} below.
However, the geodesic distance considered in our approach will turn out to be a metric (on suitable submanifolds).
To assure that the distance descends to a distance function on $\SSpace$, the pseudometric needs to satisfies the following invariance property.

\begin{definition}
 Let $d_\PSpace \colon \PSpace \times \PSpace \rightarrow [ 0 , \infty [$ be a pseudometric. 
 Then $d_\PSpace$ is called \emph{reparametrization invariant} if
 \begin{equation}\label{Eq:ReparametrizationInvariance}
        d_\PSpace(c_0, c_1) = d_\PSpace(c_0 \circ \varphi, c_1 \circ \varphi) \quad \forall \varphi \in \Diffeom.
\end{equation}
In other words $d_\PSpace$ is invariant with respect to the diagonal (right) action of $\Diffeom$ on $\PSpace \times \PSpace$.
\end{definition}

Let $[c_0], [c_1] \in \SSpace$ be equivalence classes and pick arbitrary representatives $c_0 \in [c_0]$ and $c_1 \in [c_1]$.
This allows us to define a pseudometric $d_\SSpace$ as
\begin{equation}\label{Eq:DistanceShape}
    d_\mathcal{S}([c_0], [c_1]) := \inf_{\varphi \in \Diffeom} d_\mathcal{P} (c_0, c_1 \circ \varphi).
\end{equation}

Then we obtain the following result.

\begin{lemma}\label{lem:RepresentationIndependence}
    If $d_\PSpace$ is a reparametrization invariant pseudometric on $\PSpace$,
    then the pseudometric $d_\SSpace$ on $\SSpace$, as defined in Equation \eqref{Eq:DistanceShape},
    is independent of the choice of representatives $c_0$ and $c_1$.
\end{lemma}
\begin{proof}
   Due to the infimum over $\varphi$, the distance $d_\SSpace$ is independent of the choice of the representative $c_1$.
   With respect to the choice of representative $c_0$, we easily find:
\begin{align*}
    d_\mathcal{S}([c_0 \circ \psi], [c_1]) &= \inf_{\varphi \in \Diffeom} d_\mathcal{P} (c_0 \circ \psi, c_1 \circ \varphi)\\
    &= \inf_{\varphi \in \Diffeom} d_\mathcal{P} (c_0 \circ \psi \circ \psi^{-1}, c_1 \circ \varphi \circ \psi^{-1})\\
    &= \inf_{\gamma \in \Diffeom} d_\mathcal{P} (c_0, c_1 \circ \gamma)\\
   &= d_\mathcal{S}([c_0], [c_1]).\qedhere
\end{align*}
\end{proof}

\begin{remark}\label{rem: pseudometric}
 Contrary to the finite-dimensional case, the geodesic distance of an infinite-dimensional Riemannian manifold might vanish.
 For example, it is well known that the geodesic distance with respect to the standard $L^2$-metric on $\mathcal{P}$ vanishes everywhere.
 This is the reason why we formulated the above results using pseudometrics.
 
 However, it is also known that this pathology does not occur for the class of first order Sobolev metrics.
 We refer to \cite{MR3444354} for more information and references.  
\end{remark}

Next we will show how to obtain a reparametrization invariant geodesic distance function from a certain first order Sobolev metric.
Thus its geodesic distance will turn out to be metric (on certain submanifolds of $\PSpace$).

\subsection{SRV transform for curves on a Lie group}

The main idea is to construct a well behaved mapping which allows us to pull the $L^2$ inner product on $C^\infty (I,\LieA)$ back to a Riemannian metric on a suitable submanifold of $\PSpace$.
Inspired by the approach in \cite{srivastava_shape_2011,su_statistical_2014, su_rate-invariant_2014}, we define the map 
\begin{equation}\label{Eq:SRVT}
 \begin{aligned}
{\SRVT :\text{Imm}(I,G)\rightarrow \{ q\in C^{\infty}(I,\mathfrak{g})}\, |\, q(t) \neq 0 \quad \forall t \in I \} = C^\infty (I,\LieA \setminus \{0\}),\\
    q(t)=\SRVT(c)(t):=\frac{\RightTrans_{c(t)*}^{-1}(\dot{c}(t))}{\sqrt{\|\dot{c}(t)\|}}
    \end{aligned}
\end{equation}
where the norm $\| \cdot \|$ is induced by a right invariant metric on $G$.
In particular, this entails that the norm on $\LieA$ is induced by an inner product $\langle \cdot , \cdot \rangle$ on $\LieA$.
Note that $\SRVT$ is not injective since we lose information regarding the curve's starting point.
For $\LieG = \mathbb{R}^d$ this map is known as the \emph{Square Root Velocity Transform} (SRVT) \cite{srivastava_shape_2011}. 
Hence we will also call $\SRVT$ the \emph{square root velocity transform} if $\mathcal{P}$ is a space of mappings with values in an arbitrary Hilbert Lie group.

Let us first note some properties of the SRVT.

\begin{lemma}\label{lem: SRVT:prop}
 The SRVT $\SRVT$ is 
 \begin{compactenum}
  \item \emph{equivariant with respect to reparametrizations}, i.e.\ $\forall (c,\varphi) \in \PSpace \times \Diffeom$ we have $\SRVT (c\circ \varphi) = \SRVT(c) \circ \varphi \cdot \sqrt{\dot{\varphi}}$,
  \item \emph{translation invariant}, i.e.\ $\forall g \in \LieG$ and $c \in \PSpace$ we have $\SRVT (R_g \circ c) = \SRVT (c)$.
 \end{compactenum}
\end{lemma}

\begin{proof}
 \begin{compactenum}
  \item A straight forward computation yields
\[
    \SRVT(c \circ \varphi)(t) = \frac{\RightTrans_{c(\varphi(t))*}^{-1}(\dot{c}(\varphi(t))\dot{\varphi}(t))}{\sqrt{\|\dot{c}(\varphi(t))\dot{\varphi}(t)\|}} = \sqrt{\dot{\varphi}(t)} \ \frac{\RightTrans_{c(\varphi(t))*}^{-1}(\dot{c}(\varphi(t)))}{\sqrt{\|\dot{c}(\varphi(t))\|}}.
\]
 \item This follows from $\RightTrans_{c(t)g*}^{-1}(R_{g*}\dot{c}(t)) = \RightTrans_{c(t)*}^{-1}(\dot{c}(t))$ (apply \cite[38.1 Lemma]{kriegl97tcs} to the constant map $t \mapsto g$) and the definition of $\SRVT$.\qedhere
 \end{compactenum}
\end{proof}

\begin{definition}\label{defn: pseudodistance}
 Define a pseudometric on $\PSpace = \mathrm{Imm} (I,\LieG)$ via
 \begin{equation}\label{Eq:L2pseudoDistance}
    d_{\PSpace}(c_0, c_1) := \sqrt{\int_I \| q_0(t) - q_1(t) \|^2 \dif t} = d_{L^2}( \SRVT(c_0), \SRVT(c_1)),
\end{equation} where $q_i := \SRVT(c_i),\, i=0,1$.
\end{definition}

Notice that $d_{\PSpace}$ is only a pseudometric as it does not distinguish between $c\in \PSpace$ and $R_g \circ c$ (for $g\in \LieG$ by Lemma \ref{lem: SRVT:prop}). 

\begin{proposition}
    The pseudometric $d_{\PSpace}$ from Definition \ref{defn: pseudodistance} is reparametrization invariant.
\end{proposition}
\begin{proof}
In order to prove that $d_{\PSpace}$ is reparametrization invariant, we need to show that the property \eqref{Eq:ReparametrizationInvariance} holds.

Using the definition \eqref{Eq:L2pseudoDistance} and the substitution $s := \varphi(t)$, the reparametrization equivariance (see Lemma \ref{lem: SRVT:prop}) implies:
\begin{align*}
    &d_{\PSpace}(c_0 \circ \varphi, c_1 \circ \varphi) =\\
    &=\left( \, \bigintsss_I \, \norm{\sqrt{\dot{\varphi}(t)} \left( \frac{\RightTrans_{c_0(\varphi(t))*}^{-1}(\dot{c_0}(\varphi(t)))}{\sqrt{\|\dot{c_0}(\varphi(t))\|}} - \ \frac{\RightTrans_{c_1(\varphi(t))*}^{-1}(\dot{c_1}(\varphi(t)))}{\sqrt{\|\dot{c_1}(\varphi(t))\|}} \right)}^2 \dif t \, \right)^{1/2}\\
    &= \left(\, \bigintsss_I \, \norm{ \left( \frac{\RightTrans_{c_0(s)*}^{-1}(\dot{c_0}(s))}{\sqrt{\|\dot{c_0}(s)\|}} - \frac{\RightTrans_{c_1(s)*}^{-1}(\dot{c_1}(s))}{\sqrt{\|\dot{c_1}(s)\|}}\right) }^2\dif s\, \right)^{1/2}\\
    &= d_{\PSpace}(c_0, c_1).\qedhere
\end{align*}
\end{proof}

We would like to realize $d_{\PSpace}$ as the geodesic distance of a Riemannian metric on $\PSpace$ which arises by pullback with $\SRVT$.
In the classical case, of the abelian Lie group $\mathbb{R}^2$, the map $\SRVT$ reduces to the SRVT considered in \cite{bauer_constructing_2014}.
As was observed in \cite[3.4 and Remark 3.9]{bauer_constructing_2014}, the map $\SRVT$ is not even infinitesimally injective, i.e.\ the kernel of the tangent map $T_c \SRVT$ is not trivial.
Hence, pulling back the $L^2$-metric will not result in a Riemannian metric on all of $\PSpace$. 
However, the map $\SRVT$ restricts to a diffeomorphism onto a certain submanifold of $\PSpace$.
Further, on this submanifold $d_{\PSpace}$ will coincide with the geodesic distance induced by the pullback of the $L^2$-metric.  
Before we prove this, let us first introduce some auxiliary mappings and an explicit formula for the inverse of the SRVT.
Consider the scaling maps 
  \begin{equation} \label{eq: scaling}\begin{aligned}
                     \mathrm{sc} \colon C^\infty (I,\mathfrak{g} \setminus \{0\}) &\rightarrow C^\infty (I,\mathfrak{g}\setminus \{0\}),\quad   q \mapsto \left(t \mapsto \frac{q(t)}{\sqrt{\|q(t)\|}}\right)  \\
                     \mathrm{sc}^{-1} \colon C^\infty (I,\mathfrak{g} \setminus \{ 0\})&\rightarrow C^\infty (I,\mathfrak{g} \setminus \{ 0\}),\quad q \mapsto (t \mapsto q(t) \| q(t)\|)
                   \end{aligned}
  \end{equation}
Then it is easy to see that $\SRVT (c) = \mathrm{sc} \circ \delta^r (c)$ for all $c \in \PSpace$.
We will see in Lemma \ref{lem: ind:diffeo} that $\SRVT$ induces a diffeomorphism $\PSpace \cap C^\infty_* (I,\LieG) \rightarrow C^\infty (I,\LieA\setminus \{0\})$.
By abuse of notation we write $\SRVTInv := \left(\left.\SRVT\right|_{\PSpace \cap C^\infty_* (I,\LieG)}\right)^{-1}$.% a a scaled version of the evolution operator $\Gamma$, sending curves in $\LieA \setminus \{0\}$ to curves on $G$:
Before we prove that $\SRVTInv$ is smooth and provide a formula, consider the set 
\begin{displaymath}
\PSpace_* = \{ c \in \mathrm{Imm}(I, \LieG): \, c( 0 ) = e \}.
\end{displaymath}
Now $\PSpace_* = \PSpace \cap C^\infty_* (I,\LieG)$ and by \cite[1.10]{glockner_regularity_2012} $\PSpace_*$ is a closed submanifold of $\PSpace$.
Note that $\SRVTInv$ is the inverse of the SRVT when restricted to the submanifold $\PSpace_*$.

In the following, we restrict our investigation to $\PSpace_*$, i.e.\ we consider only curves starting at the identity element in $\LieG$.
We remark that this is only a mild restriction as the group operations in $\LieG$ allow us to transport any smooth curve $c \colon I \rightarrow \LieG$ to the smooth curve $c(t)\cdot c(0)^{-1}$ starting at the identity.

\begin{lemma}\label{lem: ind:diffeo}
 The scaling maps \eqref{eq: scaling} are smooth diffeomorphisms.
 Moreover, the SRVT is smooth and induces a diffeomorphism $\PSpace_*\rightarrow C^\infty (I,\LieA \setminus \{ 0 \})$ with inverse
\begin{align}\label{Eq:SRVTInv}
\begin{split}
    &\SRVTInv \colon C^\infty(I, \LieA \setminus \{0\}) \rightarrow C_*^\infty(I, \LieG)\\   
    &\SRVTInv(q)(t) = \Gamma \circ \mathrm{sc}^{-1}  (q) = c(t),\quad \text{where} \,\, \frac{\partial c}{\partial t}=\RightTrans_{c(t)*}(q(t)\|q\|),\,\, c(0)=e.
\end{split}
\end{align}.
\end{lemma}

\begin{proof}
Since the norm $\|\cdot\|$ is induced by an inner product, it is smooth away from $0$ (cf. the discussion in \cite[Proposition 13.14]{kriegl97tcs}).
 In particular, this entails that the scaling maps \eqref{eq: scaling} are smooth by \cite[Theorem 42.13]{kriegl97tcs}. 
 
 To see that $\SRVT$ is a diffeomorphism, observe that $\Gamma \colon C^{\infty}(I,\mathfrak{g}) \rightarrow C_{*}^{\infty}(I,G)$  is a (smooth) diffeomorphism by \cite[Theorem A]{glockner_regularity_2012}.
 Its inverse is the right logarithmic derivative $\delta^r \colon C^\infty_* (I,\LieG) \rightarrow C^\infty (I,\LieA)$.
 Hence $\SRVT = \mathrm{sc}^{-1} \circ \delta^r|_{\mathrm{Imm} (I,\LieG) \cap C_{*}^{\infty}(I,G)}$ is a diffeomorphism whose inverse clearly is $\SRVTInv$.
 As $\SRVTInv = \Gamma \circ \mathrm{sc}$ is also smooth, the assertion follows.
\end{proof}

The right action $\PSpace \times \Diffeom \rightarrow \PSpace, \quad (c,\varphi) \mapsto c \circ \varphi$ restricts to an action on $\PSpace_*$. 
We can now define a reparametrization invariant (pseudo)metric\footnote{We will see later see Theorem \ref{thm: pbmetric} and Theorem \ref{thm: globaldist} that the pseudometric is the geodesic distance with respect to a first order Sobolev metric if dim $\LieA > 2$. Hence in these case we actually obtain a metric on $\PSpace_*$.} on $\PSpace_*$ as follows.
\begin{definition}
 Restrict $d_\PSpace$ to a pseudometric $d_{\PSpace_*}$ on $\PSpace_*$, i.e.\ from \eqref{Eq:L2pseudoDistance} we derive
 \begin{equation}\label{Eq:L2Distance}
    d_{\PSpace_*}(c_0, c_1) := \sqrt{\int_I \| q_0(t) - q_1(t) \|^2 \dif t} = d_{L^2}( \SRVT(c_0), \SRVT(c_1)),
\end{equation} where $q_i := \SRVT(c_i),\, i=0,1$.
\end{definition}

%We will see in the next section that $d_{\PSpace_*}$ is intimately connected to the $L^2$ inner product on the image of the SRVT transform.

\subsection{Riemannian geometry of the SRV transform}
There is a geometric interpretation of the distance function $d_{\PSpace_*}$ obtained via the SRVT that motivates our choice of approach. 
In the present section we explore this interpretation in the context of Riemannian geometry on spaces of curves.

Consider the space of curves $C^\infty (I,\LieA)$ with the $L^2$-metric 
  \begin{displaymath}
   \langle h,g \rangle_{L^2} := \int_I \langle h(t),g(t)\rangle \dif t.
  \end{displaymath}
This metric defines a (weak) Riemannian metric\footnote{Here the term ``weak'' means that the Riemannian metric does not determine the topology on $T_c C^\infty (I, \LieA) = C^\infty (I, \LieA)$.} on $C^\infty (I, \LieA)$.
Moreover, since the image of the SRVT is an open subset of $C^\infty (I,\LieA)$, the Riemannian metric restricts to a Riemannian metric on the image of the SRVT. 
Now we exploit that the SRVT is a diffeomorphism (Lemma \ref{lem: ind:diffeo}), to obtain a pullback metric on $\PSpace_*$ whose geodesic distance will turn out to be $d_{\PSpace_*}$.
Before we prove this, let us derive a formula for the pullback metric.

Recall that $\delta^r c  := R_{c^*}^{-1} (\dot{c})$ is the right logarithmic derivative whose tangent map at $c$ we denote by $T_c \delta^r$.

\begin{theorem}\label{thm: pbmetric}
 Let $c \in \PSpace_*$ and consider $v,w \in T_c \PSpace_*$, i.e.\ $v ,w \colon I \rightarrow TG$ are curves with $v(t),w(t) \in T_{c(t)}\LieG$.
 The pullback of the $L^2$-metric on $C^\infty(I, \LieA \setminus \{0\})$ under the SRVT to the manifold of immersions $\PSpace_*$ is given by: 
  \begin{equation}\label{Eq:ElasticMetric} \begin{aligned}
   G_c (v,w) = \int_I \frac{1}{4} &\left\langle D_s v, u_c\right\rangle \left\langle D_s w, u_c\right\rangle  \\ &+ \left\langle D_s v-u_c \left\langle D_s v,u_c \right\rangle , D_s w-u_c\left\langle D_s w,u_c \right\rangle \right\rangle  \dif s,
   \end{aligned}
  \end{equation}
  where $D_s v := T_c \delta^r (v)/ \norm{ \dot{c} }$, $u_c := \delta^r(c)/\norm{\delta^r (c)}$ is the unit tangent vector of $\delta^r (c)$ and $\dif s = \norm{\dot{c}(t)} \dif t$.
 Consequently, the pullback of the $L^2$-norm is given by
  \begin{displaymath}
       G_c (v,v) = \int_I \frac{1}{4} \left\langle D_s v, u_c\right\rangle^2 + \left\|D_s v-u_c\left\langle D_s v, u_c \right\rangle\right\|^2 \dif s.
       \end{displaymath}
\end{theorem}

\begin{proof}
We have to compute the tangent map of the SRVT at a curve $c$ in the direction of a vector field $v$ along $c$. 
Recall that $\SRVT = \text{sc} \circ \delta^r$, whence the chain rule implies $T_c \SRVT (v) = T \text{sc} \circ T_c \delta^r (v)$. Hence setting $z_v := T_c \delta^r (v) \in T_{\delta^r (c)} C^\infty (I, \LieA \setminus \{0\})$, it suffices to compute $T_{\delta^r (c)} \text{sc} (z_v)$.
 
By Proposition \ref{prop: LGP:prop}, there is a smooth map $\gamma \colon ]-\varepsilon , \varepsilon[ \times I \rightarrow \LieA \setminus \{0\}$ with $\gamma(0,t) = \delta^r (c)(t)$ and $\tpd{}{x} \gamma (x,t) = z_v(t)$.
%Let us first compute the tangent map of the right logarithmic derivative $\delta^r$. 
  %In passing to the last line, we have used that $c$ does not depend on $\varepsilon$ and the formula for the tangent multiplication in $TG$ (see \cite[Lemma 9.1.6 (a)]{hilgert12sag}).
%  To simplify the notation, define $v^r (t) := R^{-1}_{c(t)^*}(v(t))$, then \eqref{eq: derivative} yields  
%  \begin{align*}
%   z_v(t) &:= T_c \delta^r (v) (t) =  \frac{\partial}{\partial t} v^r(t) + \left[ v^r(t), \delta^r c(t) \right] .
%  \end{align*}
% From the formula for $T_c\delta^r$, we derive the tangent map of the SRVT as
Thus the tangent map of $\text{sc}$ can be computed as follows.
\begin{align*}
 T_c \SRVT (v) &=  T_{\delta^r (c)} \text{sc} (z_v)(t) = \left.\frac{\partial}{\partial x}\right|_{x =0}\text{sc} (\gamma (x,t)) 
			    = \left.\frac{\partial}{\partial x}\right|_{x =0} \frac{\gamma (x,t)}{\sqrt{\|\gamma (x,t)\|}}\\
			    &= \|\delta^r c (t))\|^{-\frac{1}{2}} z_v(t) -\frac{1}{2} \|\delta^r c (t))\|^{-\frac{5}{2}} \left\langle z_v(t), \delta^r c (t) \right\rangle \delta^r c (t).
\end{align*}
%\begin{align*}
%    T_c \SRVT (v) &= \left.\frac{\partial}{\partial \varepsilon}\right|_{\varepsilon =0} \SRVT (\gamma (\varepsilon ,t)) 
%			    = \left.\frac{\partial}{\partial \varepsilon}\right|_{\varepsilon =0} \frac{\RightTrans_{\gamma (\varepsilon,t)*}^{-1} \left(\frac{\partial}{\partial t} \gamma (\varepsilon,t) \right)}{\sqrt{\|\RightTrans_{\gamma (\varepsilon,t)*}^{-1} \left(\frac{\partial}{\partial t} \gamma (\varepsilon,t) \right)\|}}\\ 
%			    &= \|\delta^r c (t))\|^{-\frac{1}{2}} z_v(t) -\frac{1}{2} \|\delta^r c (t))\|^{-\frac{5}{2}} \left\langle z_v(t), \delta^r c (t) \right\rangle \delta^r c (t).
%\end{align*}
Set $z_w := T_c \delta^r (w)$ and substitute into the definition of the pullback metric these formulae. Then we simplify the expression as follows:
    \begin{equation}\begin{aligned} \label{eq: pullback1}
        &G_c(v, w) = \left< T_c \SRVT (v), T_c \SRVT (w) \right>_{L^2} \\
        =& \int_I \| \delta^r c (t) \|^{-1} \left\langle z_v(t) , z_w(t)  \right\rangle - \frac{3}{4} \| \delta^r c (t) \|^{-3} \left\langle z_v(t) , \delta^r c(t) \right\rangle \left\langle z_w(t) , \delta^r c(t) \right\rangle \dif t.
    \end{aligned}
    \end{equation}
 We set $u_c (t) := \frac{\delta^r c(t)}{\|\delta^r c (t)\|}$ and notice that $\norm{u_c(t)} =1$.
 Inserting these identities in \eqref{eq: pullback1} we can simplify $G_c (v,w)$ as follows
 \begin{align*}
	    &\int_I \| \delta^r c (t) \|^{-1}( \left\langle z_v(t) , z_w(t)\right\rangle - \frac{3}{4} \| \delta^r c (t) \|^{-2}\left\langle z_v(t) , \delta^r c(t) \right\rangle \left\langle z_w(t) , \delta^r c(t) \right\rangle) \dif t \\
	    &= \int_I \| \delta^r c (t) \|^{-1}\left( \left\langle z_v(t) , z_w(t) - \frac{3}{4} u_c (t) \left\langle z_w(t) , u_c (t) \right\rangle\right\rangle\right) \dif t\\
	    &= \int_I \| \delta^r c (t) \|^{-1}\left( \frac{1}{4} \langle z_v(t),  u_c (t)\rangle \langle z_w (t),u_c(t)\rangle + \right. \\
	    &\hspace{1cm} \left.\phantom{\frac{1}{4}}+ \langle z_v(t)-u_c(t)\langle z_v(t),u_c(t)\rangle , z_w(t)-u_c(t)\langle z_w(t),u_c(t)\rangle \rangle\right) \dif t\\
	    &= \int_I \| \delta^r c (t) \|\left( \frac{1}{4} \left\langle  \frac{T_c \delta^r (v)(t)}{ \| \delta^r c (t) \|},  u_c (t)\right\rangle \left\langle  \frac{T_c \delta^r (w)(t)}{ \| \delta^r c (t) \|},  u_c (t)\right\rangle  +\right.\\
	   &\left.+ \left\langle \frac{z_v(t)}{ \| \delta^r c (t) \|}-u_c(t)\left\langle \frac{z_v(t)}{ \| \delta^r c (t) \|},u_c(t)\right\rangle , \frac{z_w(t)}{ \| \delta^r c (t) \|}-u_c(t)\left\langle \frac{z_w(t)}{ \| \delta^r c (t) \|},u_c(t)\right\rangle \right\rangle\right) \dif t
 \end{align*}
 In passing from the second to the third line we have used the trivial identity 
 $$z_w(t) = u_c (t) \langle z_w(t),u_c(t)\rangle + (z_w(t)-u_c(t) \langle z_w(t),u_c(t)\rangle).$$
 We will now exploit that the norm and the Riemannian metric on the tangent spaces are right invariant, i.e.\ they are invariant under right translation.
 This entails $\norm{\delta^r (c)} = \norm{\dot{c}}$, whence $\dif s =  \norm{\dot{c}(t)} \dif t = \norm{\delta^r c(t)} \dif t$.
 Substitute this together with the identities for $z_v,z_w$ and $u_c$ to obtain 
 \begin{align*}
   G_c (v,w) = \int_I \frac{1}{4} &\left\langle D_s v, u_c\right\rangle \left\langle D_s w, u_c \right\rangle  \\ &+ \left\langle D_s v-u_c \left\langle D_s v,u_c \right\rangle , D_s w-u_c\left\langle D_s w,u_c \right\rangle \right\rangle \dif s.
  \end{align*}
 Thus \eqref{Eq:ElasticMetric} holds and the formula for the norm follows directly by specialization. 
\end{proof}

\begin{remark}
 For computations in the rest of the paper, we will need neither the explicit form of the pullback metric on $\mathcal{P}_*$ nor the norm computed in Theorem \ref{thm: pbmetric}. 
 The idea is to use the SRVT to relegate all questions concerning the metric to $(C^\infty (I,\mathfrak{g}\setminus \{0\}),\langle\cdot,\cdot \rangle_{L^2})$.
\end{remark}

 \begin{remark} 
 Theorem \ref{thm: pbmetric} generalises \cite[Theorem 4.2]{bauer_constructing_2014} (for the parameter values $a =1 = 2b$): 
 Viewing the vector space $\mathbb{R}^2$ as an abelian Lie group, our approach recovers the constructions of the pullback metric in the vector space case. 
 
 Notice however that we take a slightly different (but equivalent) perspective on the shape spaces involved in the construction: 
 In \cite{bauer_constructing_2014} the unparametrized curves are modelled as the quotient of $\mathrm{Imm} (I,\mathbb{R}^2)$ modulo translations.
 Picking representatives for each class, one can show that on the level of infinite-dimensional manifolds this yields the same concept of (unparametrized) curves as our approach.
 Moreover, as the SRVT is translation invariant (Lemma \ref{lem: SRVT:prop}), the two constructions yield the same Riemannian manifold (as already observed in \cite{bauer_constructing_2014}).
 \end{remark}

The pullback Riemannian metric constructed in Theorem \ref{thm: pbmetric} defines a Riemannian metric on the space of parametrized curves $\PSpace_*$. 
For curves which take their value in $\mathbb{R}^n$ the two terms in the integral can be seen as measuring bending and stretching deformations, respectively \cite{srivastava_shape_2011, bauer_constructing_2014}.
Therefore, metrics of the form \eqref{Eq:ElasticMetric} are known as \emph{elastic metrics}.
% Similarly, if we take $G$ to be the special orthogonal group $\SOT$, as we will in our computer animation applications, the projection onto the tangent vector measures changes in angular velocity, while the normal component measures changes in the axis of rotation.
\medskip

Note that the formula for the Riemannian metric is given in terms of a tangent map of the right-logarithmic derivative. 
%We have not yet computed this tangent map as often in applications (see e.g.\ Section \ref{Sec:NumericalResults}) the tangent map is not needed in the computations. 
We here give an explicit formula for this map. A proof can be found in Appendix \ref{App: linearLie}.

\begin{proposition}\label{Prop: trlog}
 Let $c \in C^\infty (I,\LieG)$ and $v \in T_c C^\infty (I,\LieG)$, i.e.\ $v \in C^\infty (I,TG)$ with $v(t) \in T_{c(t)} G,\ \forall t \in I$. 
 Then writing $\left[ \cdot, \cdot \right]$ for the Lie bracket in $\LieA$ we have
  \begin{displaymath}
   T_c \delta^r (v) (t) = \frac{\dif}{\dif t} \left(R_{(c(t))^*}^{-1}(v(t)) \right) + \left[R_{(c(t))^*}^{-1}(v(t)), \delta^r (c) (t) \right].
  \end{displaymath}
\end{proposition}

In what follows, we discuss the Riemannian geometry of the $L^2$-metric on the image of the SRVT. 
This will shed light on the geometry induced by the pullback metric on $\PSpace_*$.
Geodesics in the image of $\SRVT$ are just restrictions of geodesics in $C^\infty(I, \LieA)$ (with respect to the $L^2$-metric) to $\mathrm{Im}\ \SRVT = C^\infty(I, \LieA\setminus \{0\})$.
These geodesics correspond to geodesics in $\PSpace_*$ with respect to the pullback metric \eqref{Eq:ElasticMetric}.
As we are interested in the geodesic distances on $\PSpace_*$, let us first try to understand geodesics on $\text{Im } \SRVT \subseteq C^\infty (I, \LieA)$. 

\begin{proposition}\label{prop: curvature}
 Consider $C^\infty (I,\mathfrak{g})$ with the weak Riemannian structure induced by the $L^2$-inner product $\langle f, g \rangle_{L^2} := \int_I \langle f(t),g(t)\rangle \dif t$.
 \begin{compactenum}
  \item The space $(C^\infty (I,\mathfrak{g}), \langle\cdot,\cdot\rangle_{L^2})$ is flat in the sense of Riemannian geometry.
  \item The open subset $C^\infty (I,\LieA\setminus \{0\}) \subseteq C^\infty (I,\LieA)$ is also flat.
        Furthermore, for $\text{\upshape dim } \LieA >1$ there exist points in $C^\infty (I,\LieA\setminus \{0\})$ not connected by a minimizing geodesic. 
 \end{compactenum}
\end{proposition}

\begin{proof}
 % Recall first, that $\langle \cdot , \cdot \rangle_{L^2}$ is indeed a weak Riemannian metric iis continuous recall that $C^\infty (I,\LieA) \times C^\infty (I,\LieA) \cong C^\infty (I,\LieA \times \LieA)$.
 % Now as $\langle \cdot, \cdot \rangle \colon \LieA \times \LieA \rightarrow \mathbb{R}$ is smooth, so is $\kappa \colon C^\infty (I,\LieA \times \LieA) \rightarrow C^\infty (I,\LieA) , (f_1,f_2) \mapsto (t \mapsto \langle f_1(t),f_2(t)\rangle)$.
 % As integration $I \colon C^\infty (I,\mathbb{R}) \rightarrow \mathbb{R}, f \mapsto \int_I f(t)\dif t$ is linear and continuous with respect to the compact open $C^\infty$-topology, $I$ is smooth. 
 % We see that $\langle \cdot , \cdot \rangle_{L^2} = I \circ \kappa$ is smooth as a composition of smooth maps, whence $\langle \cdot , \cdot \rangle_{L^2}$ defines a Riemannian metric on the vector space $C^\infty(I,\LieA)$
  \begin{compactenum}
   \item The Riemannian metric is the inner product of a vector space. Hence $C^\infty (I,\mathfrak{g})$ is flat as all derivatives of the Riemannian metric with respect to the base point vanish. 
   % We endow $C^\infty (I,\LieA)$ with the compact open topology, then both inclusions become continuous (cf.\ \cite[Definition I.5.1]{MR2261066}, whereas the )and the Hilbert space product of $L^2 (I,\LieA)$.
%   Obviously the Hilbert space product of $L^2 (I,\LieA)$Endowing $C^\infty (I,\LieA)$ with the compact open $C^\infty$-topology and $L^2 (I,\LieA)$ with the canonical  Hilbert space topology, we see that $\iota$ is continuous (the map factors through $C^\infty (I,$ induced by the $L^2$ inner product)It is easy to see that the $L^2$-inner product is continuous with respect to the compact-open $C^\infty$-topology on $C^\infty (I, \LieA)$. 
%   %Hence $$The assertion seems to be well known, unfortunately, we were not able to track down an explicit proof in the literature.
%    We here adapt ideas from \cite[Theorem 9.1]{MR0271984} to this proof: 
%    Let $\kappa$ be the affine connection associated to the Riemannian metric on $\LieA$, i.e.\ as $\LieA$ is a Hilbert space, identifying tangent bundles leads to $\kappa \colon T^2\LieA = \LieA^4 \rightarrow T\LieA = \LieA^2, (x,y,v,w) \mapsto (x,w)$. Since $TC^\infty (I,\LieA) \cong C^\infty(I,T\LieA)$, we obtain an affine connection on $C^\infty (I,\LieA)$ as $\kappa_* \colon C^\infty (I, T(T\LieA)) \rightarrow C^\infty (I,T\LieA),\quad f \mapsto \kappa \circ f$.
%    Now one computes as in \cite[Theorem 9.1]{MR0271984} that $\kappa_*$ is the affine connection associated to the $L^2$-metric.
%    Using this explicit description the assertion follows easily. 
  \item Note that $C^\infty (I,\LieA\setminus \{0\})$ is an open subset of $C^\infty (I,\LieA)$, whence flat. 
  For $\text{dim } \LieA > 1$ the set $\LieA \setminus \{0\}$ is connected but not convex, whence there are points not connected by minimizing geodesics, e.g.\ for $c \in C^\infty (I,\LieA\setminus \{0\})$ the minimizing geodesic connecting $c$ and $-c$ is not contained in $C^\infty (I,\LieA\setminus \{0\})$ (since $c(t) + \tfrac{1}{2}(-c(t)-c(t))=\tfrac{1}{2} c(t) +\tfrac{1}{2} (-c(t)) = 0$).
  \qedhere
  \end{compactenum}
\end{proof}

 As the Riemannian structure of $C^\infty (I,\LieA \setminus \{0\})$ is induced by $C^\infty (I,\LieA)$, the $L^2$-distance locally describes the geodesic distance. 
 Thus the distance function \eqref{Eq:ElasticMetric} locally describes the geodesic distance with respect to the pullback Riemannian metric on $\PSpace_*$.
 Since minimizing geodesics need not exist between points $f,g \in C^\infty (I,\LieA \setminus \{0\})$, the geodesic distance might be strictly greater than the $L^2$-distance if $f(s) +t (f(s)-g(s)) =0$.
 
 However, we will prove that, at least if the dimension of the Lie algebra is large enough, the $L^2$-distance coincides with the geodesic distance (see Theorem \ref{thm: globaldist}).  
 To prove that the geodesic distance on $C^\infty (I,\LieA \setminus \{0\})$ coincides with the $L^2$-distance, we have to approximate the minimizing geodesic by paths in $\LieA \setminus \{0\}$.
% The minimizing geodesic $q \colon [0,1] \rightarrow C^\infty (I,\LieA)$ yields a smooth map $q^\wedge \colon [0,1] \times I \rightarrow \LieA , q^\wedge(t,s) := q(t)(s)$.
% We want to think of $q^\wedge$ as an immersion and of its image as an immersed surface in $\LieA$. 
% The main idea here is that if the surface $\mathrm{Im}\ q^\wedge$ passes through the origin in $\LieA$, we perturb it in a normal direction to move it away from $0$.
% However, for arbitrary curves $f$ and $g$, $q^\wedge$ need not be an immersion.
% This can be remedied by replacing $f$ and $g$ with elements (arbitrarily close with respect to the $L^2$-distance) such that $q^\wedge$ is an immersion.
% The details of this construction and the proof of Theorem \ref{thm: globaldist} below are quite involved and technical.
 To avoid a lengthy exposition at this time, we have relegated these details to Appendix \ref{App: distance}.

\begin{theorem}\label{thm: globaldist}
  If $\text{\upshape dim } \LieA > 2$ then the geodesic distance of $C^\infty (I,\LieA\setminus \{0\})$ is globally given by the $L^2$-distance.  
  In particular, in this case the geodesic distance of the pullback metric \eqref{Eq:ElasticMetric} on $\PSpace_*$ is given by the distance function \eqref{Eq:L2Distance}.
\end{theorem}

\begin{remark}
 Notice that Theorem \ref{thm: globaldist} entails that the pseudometric $d_{\PSpace_*}$ defined via \eqref{Eq:L2Distance} is non-degenerate for $\text{dim } \LieA >2$, i.e.\ $d_{\PSpace_*} (c,\tilde{c}) >0$ if $c \neq \tilde{c}$. 
 Thus $d_{\PSpace_*}$ is a metric on $\PSpace_*$. In particular, we will say that $d_{\PSpace_*}$ defines a distance function on $\PSpace_*$ in this case.
\end{remark}

The crucial observation here is that one obtains a distance on $\PSpace_*$. 
This enables us in Section \ref{Sec:NumericalResults} to compute distances between curves and to deform curves into each other along geodesic paths.
%Measuring distances between curves in $\PSpace_*$ by applying the SRVT followed by the $L^2$ inner product as in $\eqref{Eq:L2Distance}$ then corresponds to determining the geodesic distance between the curves in $\PSpace_*$ under the pullback metric.
The relation of geodesic distance and $L^2$-distance has previously been used in animation classification tasks in $\mathbb{R}^d$ for $d > 2$ e.g.\ in \cite{eslitzbichler_modelling_2014}.

Finally, we can follow the argument given in \cite[Theorem 6.1]{bauer_constructing_2014} to derive information on the curvature of the shape space:

\begin{corollary}
 The curvature of the space $\SSpace_* := \PSpace_* / \Diffeom$ with the Riemannian metric induced by the pullback metric \eqref{Eq:ElasticMetric} is non-negative.
\end{corollary}

\begin{proof}
 The Riemannian structure on $\SSpace_*$ is induced by the one on $\PSpace_*$, i.e.\ the canonical quotient map $\PSpace_* \rightarrow \SSpace_* = \PSpace_*/\Diffeom$ is a Riemannian submersion.
 Hence we can apply the O'Neil curvature formula (see e.g.\ \cite[Theorem 3.20]{MR0458335}) for the sectional curvature of the quotient $\SSpace_*$. 
 For orthonormal vector fields $X,Y$ on $\SSpace_*$ this yields: 
  \begin{displaymath}
   K_{\SSpace_*} (X,Y) = K_{\PSpace_*} (\tilde{X}, \tilde{Y}) + \frac{3}{4}\norm{\left[\tilde{X}, \tilde{Y}\right]^\mathrm{vert}}^2
  \end{displaymath}
 Here $\tilde{X}$ and $\tilde{Y}$ are horizontal lifts of $X$ and $Y$ to $\PSpace_*$ and $\left[\tilde{X}, \tilde{Y}\right]^\mathrm{vert}$ denotes the vertical projection.
 As $\PSpace_*$ is flat by Proposition \ref{prop: curvature}, the curvature $K_{\PSpace_*}$ vanishes, whence $K_{\SSpace_*}$ is non-negative.
\end{proof}

\subsection{Distance as an optimization problem}
In the last section we acquired a distance function on $\PSpace_*$ for dim $\LieA >2$.
Using this distance function, we can calculate distances in the shape space $\SSpace_* := \PSpace_* / \Diffeom$\footnote{Recall that in the beginning of the Section 3 we set out to construct a distance function for the shape space $\SSpace = \PSpace / \Diffeom$. As explained in the last section, this is not possible within the SRVT-framework. Hence we have to use the smaller space $\SSpace_*$ (which can be identified with a subset of $\SSpace$).} by solving the optimization problem
\begin{equation}\label{Eq:OptimProblem}
    d_{\SSpace_*}( [c_0], [c_1] ) := \inf_{\varphi \in \Diffeom} \left( \int_I \| q_0(t) - q_1(\varphi(t))\sqrt{\dot{\varphi}(t)}\|^2 \dif t \right)^{1/2},
\end{equation} where $q_i := \SRVT(c_i),\, i=0,1$.

When computing $d_{\SSpace_*}$ we therefore need to perform an optimization over the diffeomorphism group $\Diffeom$ (see for example \cite{lahiri_precise_2015} for more details and an extension to piecewise linear curves instead of immersions).

In practice, one of two different algorithms is used to solve this optimization problem: either a gradient descent based approach or a dynamic programming (DP) algorithm.
For our numerical experiments and applications in Section \ref{Sec:NumericalResults}, we have used DP, which constructs a piecewise linear approximation of the optimal reparametrization $\varphi$ in \eqref{Eq:OptimProblem}.
See \cite{sebastian_aligning_2003, bauer_landmark-guided_2015} and references therein for more information on the use of DP for shape analysis.

\section{Closed curves}\label{Sec:ClosedCurves}
In various applications, we are particularly interested in closed curves.
For example, in object recognition, closed planar curves can be used to represent outlines of objects.
In computer animation, closed curves are cyclic animations that can be repeated multiple times with no visually noticeable discontinuities.

Here we will derive a method to calculate a closed curve approximation of an existing open curve.
Again, we will be working with SRV transformed representatives $q$ of curves $c \in \PSpace_*$.
We denote the image sets via the $\mathcal{R}$ transform of the open and closed immersions on $G$ by %representatives of open and closed curves by
\begin{align*}
    \mathcal{C}^o &:= \mathcal{R}\big(\mathrm{Imm}(I,G)\big) = C^\infty (I, \LieA\setminus \{0\})=\{ q\in C^{\infty}(I,\mathfrak{g}) \, |\, q(t) \neq 0 \quad \forall t \in I \} \nonumber\\
    \mathcal{C}^c &:= \mathcal{R}\big(\{ c\in \mathrm{Imm}(I,G)  \, | \, c(0)=c(1)=e \}\big),
    %=\{ q\in \mathcal{C}^o \,:\, \SRVTInv(q)(1)=e\},
\end{align*} respectively, i.e., we will be working exclusively with parametrized curves, and not with shapes.
Notice that the open curves are an open subset of the locally convex space $C^\infty (I,\LieA)$.

A curve $c$ in the set of closed curves $\mathcal{C}^c$ will in general only be $C^0$-closed.
This means, we do not require that the derivatives $\frac{\dif}{\dif t} c(0)$ and $\frac{\dif}{\dif t} c(1)$ (or higher derivatives at the closing points) coincide.
Hence the closed curves considered here admit ``corners'' at the closing points and should not be confused with smooth loops in $\LieG$, i.e.\ smooth maps from the unit circle $\mathbb{S}^1$ to $\LieG$.

Consequently, the closed curves computed via the methods in this section will only be $C^0$-closed.
Though a higher order closing might be desirable, for the applications we have in mind the closing of the curves is sufficient.

\subsection{Closed curves as a closed submanifold}
We will now prove that $\mathcal{C}^c$ is a closed submanifold of the open curves $\mathcal{C}^o$ and thus a closed submanifold of $C^\infty (I,\LieA)$.
Since $\SRVT$ is a diffeomorphism, this entails that $\{ c\in \mathrm{Imm}(I,G)  \, | \, c(0)=c(1)=e \} = \SRVT^{-1} (\mathcal{C}^c)$ is a closed submanifold of $\PSpace_* = \{ f \in \mathrm{Imm}(I,\LieG) \mid f(0) = e\}$.

The basic idea is now to construct $\mathcal{C}^c$ as the preimage of a closed submanifold under a submersion.
To this end, consider the point evaluation map
\begin{align*}
 \mathrm{ev}_1 \colon C^\infty (I,G) &\rightarrow G,\quad  f \mapsto f(1), \quad \text{ (evaluation in } 1).
\end{align*}
It can be observed that
\begin{displaymath}
 \mathcal{C}^c=r^{-1}(e),\qquad r\colon \mathcal{C}^o\rightarrow G,\quad r:=\mathrm{ev}_1\circ \Gamma \circ \mathrm{sc},
\end{displaymath}
where $\mathrm{sc}$ denotes the scaling map from \eqref{eq: scaling} and we note that $r(q)=\mathcal{R}^{-1}(q)(1).$

\begin{proposition}\label{prop: curves:smfd}
 The map $r \colon \mathcal{C}^o \rightarrow G$ is a submersion.
 Hence $\mathcal{C}^c=r^{-1}(e)$ is a closed submanifold of finite codimension in $\mathcal{C}^o \subseteq C^\infty(I,\mathfrak{g})$.
\end{proposition}

\begin{proof}
 Let us first establish the smoothness of $r$. 
 In the proof of Lemma \ref{lem: ind:diffeo} we have already remarked that $\Gamma$ and $\mathrm{sc}$ are (smooth) diffeomorphisms. 
 Further, $\mathrm{ev}_1$ is smooth as a consequence of \cite[Theorem 42.13]{kriegl97tcs} whence it restricts to a smooth map on the closed submanifold $C_{*}^{\infty}(I,G)$.
 By abuse of notation we will denote the induced map on the submanifold also by $\mathrm{ev}_1$.
 We conclude that $r$ is smooth and in particular it will be a submersion if and only if $\mathrm{ev}_1$ is a submersion.
  
 For general locally convex (infinite-dimensional) manifolds a map is called submersion (see \cite{glockner15fos}) if for every point in its domain we can find submersion charts, i.e.\ charts around the point and its image which turn the map locally into a projection.\footnote{In the case of a finite-dimensional Lie group $G$, \cite[Theorem A]{glockner15fos} asserts that a sufficient condition for $\mathrm{ev}_1 \colon C^\infty_* (I,G)\rightarrow G$ to be a submersion is that its differential at every point is surjective. This is easily verified.}
 We will now construct submersion charts for $c \in C^\infty_* (I, G)$. 
 To this end, let $\psi \colon G \supseteq V \rightarrow W \subseteq \mathfrak{g}$ be a chart for $G$ around the identity. 
 Then 
 \begin{displaymath}
  \psi_* \colon C^\infty_* (I,G) \cap C^\infty (I,V) \rightarrow \{h \in C^\infty (I,\LieA) \mid h(0)=0\} \cap C^\infty (I,W) ,\quad  f \mapsto \psi \circ f
 \end{displaymath}
 is an identity chart for $C^\infty_* (I,G)$.
 Since $C^\infty_* (I, G)$ is a Lie group, translating $\psi_*$ by $c$ yields a chart centered at $c$.
 Denote by $\widetilde{\mathrm{ev}}_1 \colon \{h \in C^\infty (I,\LieA) \mid h(0)=0\}  \rightarrow \LieA, h \mapsto h(1)$ the point evaluation in $1$.
 We compute for $v \in \{h \in C^\infty (I,\LieA) \mid h(0)=0\}$
 \begin{displaymath}
  \mathrm{ev}_1 \circ L_{c} \circ \psi^{-1} (v) =  \mathrm{ev}_1 \circ L_{c} (\psi^{-1} \circ v) = c(1) \psi^{-1} (v(1)) = (L_{c(1)} \circ \psi^{-1}) \widetilde{\mathrm{ev}}_1 (v) . 
 \end{displaymath}
 As $L_c, L_{c(1)}$ and the charts $\psi$, $\psi_*$ are diffeomorphisms, it suffices to construct submersion charts for $\widetilde{\mathrm{ev}}_1$.
 To construct these charts, we split $\{h \in C^\infty (I,\LieA) \mid h(0)=0\}$ non canonically:  
 \begin{align*}
  \lambda \colon \{h \in C^\infty (I,\LieA) \mid h(0)=0\}  &\rightarrow \{h \in C^\infty(I,\LieA) \mid h(0)=h(1)=0\} \times \LieA,\\
   h &\mapsto (t \mapsto h(t) -th(1), h(1))
 \end{align*}
 Since $\widetilde{\mathrm{ev}}_1$ is smooth, $\lambda$ is smooth. 
 Furthermore, $\lambda$ is a isomorphism of locally convex spaces as $\lambda^{-1} (c,q) = (t\mapsto c(t)+tq)$ is also smooth.
 Now let $\mathrm{pr}_\LieA$ be the projection onto $\LieA$ in the above product.
 Then one computes $\widetilde{\mathrm{ev}}_1 \circ \lambda = \mathrm{pr}_\LieA$.
 Hence, $\lambda$ and $\mathrm{id}_\LieA$ form a pair of submersion charts for $\widetilde{\mathrm{ev}}_1$, whence $\widetilde{\mathrm{ev}}_1$ and $\mathrm{ev}_1$ are submersions.
 We deduce that $r$ is a submersion.
 
 To prove the final assertion we use that $\mathcal{C}^c = r^{-1} (e)$ is the preimage of a point under the submersion $r$. 
 Invoking a version of the regular value theorem \cite[Theorem D]{glockner15fos} for infinite-dimensional manifolds, we deduce that $\mathcal{C}^c$ is a closed submanifold of $\mathcal{C}^o$ (whence also of $C^\infty (I,\LieA)$).
 If $\LieG$ is finite-dimensional, $\mathcal{C}^c$ is a submanifold of finite codimension.
\end{proof}

\begin{remark}
 The right Lie group action $\PSpace_* \times \Diffeom \rightarrow \PSpace_*, \quad (c,\varphi) \mapsto c \circ \varphi$ restricts to an action on the closed curves $\{c \in \mathrm{Imm} (I,\LieG)\mid c(0)=e=c(1)\}$.
 Hence the pullback Riemannian metric \eqref{Eq:ElasticMetric} induces a reparametrisation invariant Riemannian metric on the space of closed curves (equivalently on $\mathcal{C}^c$).
 We remark that $\mathcal{C}^c$ will in general not be flat in the sense of Riemannian geometry (cf.\ the computation of curvature in \cite[5.2]{bauer_constructing_2014}).
\end{remark}

\begin{remark}[The framework for infinite-dimensional Lie groups]\label{rem: frame:infty}
If we assume that $G$ is a Hilbert Lie group, the results obtained so far carry over without any changes in the proofs.
Here by Hilbert Lie group we mean a Lie group in the sense of \cite[Section 36]{kriegl97tcs} or equivalently \cite{MR2261066} modelled on a Hilbert space.

More generally, some of the results are still valid for Banach Lie groups which are modeled on a Banach space with a smooth norm away from $0$ (see \cite[Chapter 13]{kriegl97tcs} for more information on smooth norms).
Notice however that in this case complications arise. For example, in the Banach Lie group case we lack an inner product. 
Hence, one has to replace the formula for the derivative of the norm in the proof of Theorem \ref{thm: pbmetric} and it is unclear how to relate the distance $d_{\PSpace_*}$ to the geodesic distance from Riemannian geometry.
\end{remark}

We address now the problem of projecting  open curves of $ \mathcal{C}^o$ onto the submanifold of closed curves $\mathcal{C}^c$. In theory, we could define a projection from $\mathcal{C}^o$ onto $\mathcal{C}^c$ by stating a constrained minimization problem
%$$\min_{q}\, \frac{1}{2}\|q-q_0\|,\quad \mathrm{subject}\, \mathrm{to}\quad r(q)=e$$
$$\min_{q \in \mathcal{C}^c}\, \frac{1}{2}\|q-q_0\|,$$
%$$\quad \frac{1}{2}\int_{S^1}\mathcal{G}_e(q_0(t),q(t))\, \dif t.$$
where $q_0$ is the curve to be approximated. Instead of minimizing the distance from closed curves to $q_0$ in what follows we opt for minimizing the closure constraint. This approach leads to the formulation of a gradient flow whose solution is the desired closed curve.

\subsection{Projection via a gradient flow}

Given an SRV representative $q$ of an open curve in the Lie group $\LieG$, i.e., $q \in \mathcal{C}^o \setminus \mathcal{C}^c$, we will try to ``close'' the curve by enforcing (at least approximately) the closedness constraint $r(q)=\SRVTInv(q)(1)=e$.

We can measure the distance between the identity and the actual endpoint by using the functional
\begin{equation}\label{Eq:ErrorFunctional}
    \Phi:\mathcal{C}^{o}\rightarrow \mathbb{R},\qquad \Phi(q):= \dfrac{1}{2} \| \log \left( \SRVTInv( q )( 1 )\right)\|^2,
\end{equation}
where $\log$ denotes the inverse of the exponential map $\exp:\mathfrak{g}\rightarrow G$, and is defined in a neighbourhood of the identity of $G$. Notice that
$$\Phi(q) = 0 \iff q\in \mathcal{C}^c.$$
Note also, that this only measures first order continuity.%, but can easily be extended to higher order.

In many practical applications, we are satisfied with $\Phi(q) \approx 0$; particularly in the computer animation applications discussed later in Section \ref{Sec:NumericalResults} a precise enforcement of the constraint $\Phi(q) = 0$ is not necessary to achieve visually pleasing results.
This leads us to the idea of closing curves by minimizing the functional $\Phi$, and since $\mathcal{C}^o$ is an open subset of a vector space, we can use a straightforward gradient descent method for solving this problem.

In order to compute the gradient of $\Phi$, we will need both the tangent map of the evolution operator $\Gamma$, for which we get \cite[38.10 Corollary]{kriegl97tcs}:
\begin{equation}\label{Eq:TGamma}
    T_q \Gamma( f )(t) = R_{ \Gamma(q)(t)* } \left(  \int_0^t\mathrm{Ad}_{ \Gamma(q)(x)^{-1} }( f(x) )\, dx \right),
\end{equation}
as well as the following tangent map for $\SRVTInv$:
\begin{equation}\label{Eq:TSRVTInv}
    T_q \SRVTInv f = T_{q \| q \|} \Gamma \left( f \| q \| + \langle f, q \rangle \dfrac{q}{\|q\|} \right).
\end{equation}
Here, $\mathrm{Ad}_g$ denotes the adjoint representation of the Lie group: ${\mathrm{Ad}_g = \LeftTrans_{g*} \circ \RightTrans^{-1}_{g*}}$.
Since $T_q r = T_q \mathrm{ev}_1 \circ \SRVTInv$ the formula \eqref{Eq:TSRVTInv} immediately yields a formula for $T_qr$.
Moreover, as the tangent map of $\mathrm{ev}_1$ can be canonically identified with an evaluation in $1$ (cf.\ proof of Proposition \ref{prop: curves:smfd}), we see that $T_qr$ is given integrating to $t=1$ in the formula \eqref{Eq:TSRVTInv}.\medskip

Furthermore, we need the tangent of the logarithm map.
To this end, let us first consider the tangent of the Lie group exponential map $T\exp \colon T\LieA = \LieA \times \LieA \rightarrow T\LieG$.
Using the right trivialisation of the tangent Lie group $TG = \LieA \rtimes_{\mathrm{Ad}} \LieG$, we have 
  \begin{equation} \label{eq: trivexp:def}
   T_u\exp (v) = R_{\exp (u)*} \circ \dexp_u (v),
  \end{equation}
for a unique (linear) map $\dexp_u: \LieA \rightarrow \LieA$ called the \emph{right trivialized tangent of the exponential map}. 
Following \cite{hausdorff_symbolische,celledoni_introduction_2014} we remark that $\dexp_u$ satisfies:
\begin{equation}\label{eq: dexp:series}\begin{aligned}
    \dexp_u (v) &= v + \dfrac{1}{2} [u, v] + \dfrac{1}{6} [u, [u, v]] + \ldots\\
    &= \sum^\infty_{k=0} \frac{1}{(k+1)!}     \mathrm{ad}^k_u v = \left. \frac{ \exp(z) - 1 } { z } \right|_{z = \mathrm{ad}_u} \mkern-36mu (v),
    \end{aligned}
\end{equation}
where $\mathrm{ad}_u$ is the derived representation of the Lie algebra $\mathrm{ad}_u v := [u,v]$.
By \cite[Remark II.5.8]{MR2261066} the formula \eqref{eq: dexp:series} holds also for infinite-dimensional Lie groups modelled on Banach spaces.
Note that $\mathrm{dexp}_u(u) = u$ and $\mathrm{dexp}_u = \delta^r (\exp)_u$ (cf.\ \cite[p.340 - 341]{MR2261066}).

By using that $\exp \circ \log = \mathrm{id}_\LieG$ in a neighbourhood of the identity in $\LieG$, \eqref{eq: trivexp:def} yields 
\begin{equation}\label{Eq:TLog}
    T_g \log = \dexp^{-1}_{ \log(g) } \left( \RightTrans_{ g^{-1} *} \right), \quad \text{for all } g \text{ in the domain of } \log. 
\end{equation}
Combining these tangent maps, we can calculate the gradient of the error functional $\Phi$ in \eqref{Eq:ErrorFunctional}.
We use the following notation for adjoints of bounded linear operators.

\begin{definition}
 Let $A \colon \LieA \rightarrow \LieA$ be a bounded linear operator and $\langle \cdot, \cdot \rangle$ be an inner product on $\LieA$. 
 Then we denote by $A^{\adj}$ the adjoint operator of $A$, i.e.\ for all $v,w \in \LieA$ we have $\langle A (v), w \rangle = \langle v, A^{\adj} (w)\rangle$.
\end{definition}

\begin{theorem}\label{thm: errorfun}
The gradient of the the error functional 
\[
    \Phi:\mathcal{C}^{o}\rightarrow \mathbb{R},\qquad \Phi(q):= \dfrac{1}{2} \| \log \left( \SRVTInv( q )( 1 )\right)\|^2,
\] 
with respect to the $L_2$ inner product is the vector field on $\mathcal{C}^o$ given by
\begin{equation}\label{Eq:GradientFlownew}
    \mathrm{grad} ( \Phi )( q ) =   \| q \| \, \theta(q) + \langle \theta(q), \dfrac{q}{ \| q \| } \rangle \, q,
\end{equation} where 
\begin{align*}
    \theta(q) &:= \mathrm{Ad}^{\adj}_{c(q)^{-1}} \, \mathrm{Ad}^{\adj}_{\LQ} \, (\dexp^{-1}_{ \log( \LQ ) })^{\adj} \,\left( \log ( \LQ ) \right) \in C^\infty(I, \LieA) \\
    c(q) &:= \SRVTInv ( q ) \in C^\infty_*(I, \LieG) \quad \text{and} \quad r(q) := \SRVTInv ( q ) ( 1 ) \in \LieG.
\end{align*}
\end{theorem}
\begin{proof}
Let $f$ be a vector field along the curve $q$.
The $L_2$ gradient of $\Phi$ is then defined by
\[
    T_q \Phi (f) = \langle \mathrm{grad} ( \Phi ) ( q ), f \rangle_{L_2} = \int_I \langle \mathrm{grad} ( \Phi ) ( q ), f \rangle \, dx,
\]
and
\[
    T_q \Phi (f) :=  \left. \dfrac{\dif}{\dif\epsilon}\right|_{\epsilon=0} \dfrac{1}{2} \left< \log( \SRVTInv( q + \epsilon f ) (1) ), \, \log( \SRVTInv( q + \epsilon f ) (1) ) \right>.
\]

By differentiation we obtain:
\begin{align*}
    T_q \Phi ( f ) &\stackrel{\hphantom{\eqref{Eq:TLog}}}{=} \left\langle T_q (\log \circ \, r)( f ), (\log \circ \, r) (q) \right\rangle \\
    &\stackrel{\eqref{Eq:TLog}}{=}  \left< \dexp^{-1}_{ \log( \LQ ) } \, \RightTrans_{ \LQ^{-1} *} \, T_q r \, (f), \log ( \LQ ) \right>\\
    &\stackrel{\hphantom{\eqref{Eq:TLog}}}{=} \left< T_q r \, (f), \RightTrans^{\adj}_{ \LQ^{-1} *} (\dexp^{-1}_{ \log( \LQ ) })^{\adj}\log ( \LQ ) \right> .
\end{align*}
Insert now the formula for the tangent maps \eqref{Eq:TGamma} and \eqref{Eq:TSRVTInv} and observe that the integral appearing is a weak integral, i.e.\ we may interchanged the integral with any of the continuous linear functionals $\langle \cdot, v\rangle$.
Hence the above formula for $T_q\Phi$ simplifies as follows:
\begin{align*}
  T_q \Phi ( f )   &=  \left< \, \LeftTrans_{\LQ*} \int_I \mathrm{Ad}_{ c(q)(x)^{-1} } \left( f(x) \| q(x) \| + \langle f(x), q(x) \rangle \dfrac{q(x)}{ \| q(x) \| } \right) \, dx, \right. \\
    &\qquad \left.  \RightTrans^{\adj}_{ \LQ^{-1} *} (\dexp^{-1}_{ \log( \LQ ) })^{\adj} \log ( \LQ ) \vphantom{\int}\right> \\
    &= \int_I \left\langle \mathrm{Ad}_{ c(q)(x)^{-1} } \left( f(x) \| q(x) \| + \langle f(x), q(x) \rangle \dfrac{q(x)}{ \| q(x) \| } \right), \right.\\
    &\qquad  \left. \LeftTrans^{\adj}_{\LQ*} \, \RightTrans^{\adj}_{ \LQ^{-1} *} \, (\dexp^{-1}_{ \log( \LQ ) })^{\adj} \, \log ( \LQ ) \vphantom{ \dfrac{q}{\norm{q}} }\right\rangle \, dx\\
    &= \int_I \left\langle f(x) \| q(x) \| + \langle f(x), q(x) \rangle \dfrac{q(x)}{ \| q(x) \| }, \theta(q)(x) \right\rangle dx\\
    &= \int_I \langle f(x) \| q(x) \|, \theta(q, x) \rangle + \left\langle f(x), q(x) \right\rangle \left\langle \dfrac{q(x)}{ \| q(x) \| }, \theta(q)(x) \right\rangle dx\\
    &= \int_I \left< f(x), \theta(q)(x) \| q(x) \| + \left\langle \theta(q)(x), \dfrac{q}{ \| q(x) \| } \right\rangle q(x) \right> dx,
\end{align*}
Reading off the gradient of $\Phi(q)$ from this expression we derive:
\begin{displaymath}
\langle\mathrm{grad}(\Phi)(q),f \rangle_{L_2} =  \left< \theta(q) \| q \| + \langle \theta(q), \dfrac{q}{ \| q \| } \rangle q, f \right>_{L_2}. \qedhere
\end{displaymath}
\end{proof}

The projection onto the space of closed curves $\mathcal{C}$ is then obtained by solving the differential equation for $u(t,\tau)$:
\begin{equation*}
\dfrac{\partial u}{\partial \tau} = -\mathrm{grad} ( \Phi )( u ), \quad u(t, 0) = q(t).
\end{equation*}

\begin{remark}
 The results of the present section on tangent maps carry over verbatim to the case of infinite-dimensional Lie groups modelled on Banach spaces.
 This is due to the fact that Banach Lie groups are locally exponential Lie groups (i.e.\ the Lie group exponential map is a local diffeomorphism near the unit) cf.\ \cite[Proposition IV.1.2]{MR2261066}.
 Moreover, the formula for the trivialised tangent map of the exponential still holds in this case by \cite[Remark II.5.8]{MR2261066}.
 In particular, this shows that Theorem \ref{thm: errorfun} remains valid for Lie groups modelled on a Hilbert space.
\end{remark}

The derivative $\dexp_u$ of the Lie group exponential appears in the formula \eqref{Eq:GradientFlownew}.
If we want to apply the formula for the gradient in computations we thus need to compute this derivative. 
However, for certain finite dimensional Lie groups, these additional computations can be avoided as the derivative vanishes in the formula.
We will prove now that in  
\begin{displaymath}
 \theta (q) =\mathrm{Ad}^{\adj}_{c(q)^{-1}} \, \mathrm{Ad}^{\adj}_{\LQ} \, (\dexp^{-1}_{ \log( \LQ ) })^{\adj}\, \left( \log ( \LQ ) \right) 
\end{displaymath}
the term $(\dexp^{-1}_{ \log( \LQ ) })^{\adj}$ vanishes if we choose the inner product induced by the Cartan-Killing form of a compact and semisimple Lie group.

\begin{definition}
 Let $\LieG$ be a finite-dimensional Lie group with Lie algebra $\LieA$. Define the \emph{Cartan-Killing form} 
 \begin{displaymath}
  \kappa_\LieA \colon \LieA \times \LieA \rightarrow \mathbb{R} , \quad \kappa_\LieA (x,y) := \mathrm{tr}(\mathrm{ad}_x \circ \mathrm{ad}_y),
 \end{displaymath}
 where we denote by $\mathrm{tr}$ the trace of a linear map.
 Recall from \cite[Remark 12.2.14]{hilgert12sag} that $\kappa_\LieA$ is a negative definite form if and only if $\LieG$ is a compact and semisimple Lie group.
\end{definition}

\begin{corollary}\label{cor: errorfun}
Let $G$ be a compact and semisimple Lie group. With respect to the $L^2$-metric constructed from the inner product $-\kappa_\LieA$ induced by the Cartan-Killing form, the gradient of the the error functional $\Phi$ \eqref{Eq:ErrorFunctional} is the vector field on $\mathcal{C}^o$ given by
\begin{equation}\label{Eq:ClosingGradient}
    \mathrm{grad} ( \Phi )( q ) =   \| q \| \, \alpha(q) + \left\langle \alpha(q), \dfrac{q}{ \| q \| } \right\rangle \, q,
\end{equation} where 
\begin{align*}
    \alpha(q) &:= \mathrm{Ad}^{\adj}_{c(q)^{-1}} \, \mathrm{Ad}^{\adj}_{\LQ} \, \left( \log ( \LQ ) \right) \in C^\infty(I, \LieA) \\
    c(q) &:= \SRVTInv ( q ) \in C^\infty_*(I, \LieG) \quad \text{and} \quad r(q) := \SRVTInv ( q ) ( 1 ) \in \LieG.
\end{align*}
\end{corollary}

\begin{proof}
 As a shorthand we define $u := \log (\LQ)$.
 Having the formula \eqref{Eq:GradientFlownew} at our disposal, it suffices to prove that the adjoint operator $(\dexp_u^{-1})^{\adj}$ fixes $u$.
 As $(\dexp_u^{-1})^{\adj} = (\dexp_u^{\adj})^{-1}$ it suffices to prove that $\dexp_u^\dagger$ fixes $u$. 
 Recall from \cite[Exercise 5.4.5]{hilgert12sag} that the operators $\mathrm{ad}_x$ are skew symmetric with respect to the Cartan-Killing form, i.e.\ $\kappa_\LieA (\mathrm{ad}_x (y),z) = \kappa_\LieA (y, -\mathrm{ad}_x (z))$ or in other words $(\mathrm{ad}_x)^{\adj} = - \mathrm{ad}_x$.
 Using the series identity \eqref{eq: dexp:series} for $\dexp_u$ together with the fact that the mapping $A \mapsto A^{\adj}$ is a continuous algebra morphism, we obtain the identity 
 \begin{displaymath}
  (\dexp_u)^{\adj}\hspace{-2pt} = \left(\sum_{k=0}^\infty \frac{1}{(k+1)!} (\mathrm{ad}_u)^k\right)^{\adj}\hspace{-2pt} = \sum_{k=0}^\infty \frac{1}{(k+1)!} (\mathrm{ad}_u^{\adj})^k \hspace{-2pt}= \sum_{k=0}^\infty \frac{1}{(k+1)!} (-\mathrm{ad}_u)^k\hspace{-2pt}.
 \end{displaymath}
 Hence $\mathrm{ad}_u (u)= 0$ implies $ (\dexp_u)^{\adj} (u) = u$ and the assertion follows.
\end{proof}

\section{Numerical Results}\label{Sec:NumericalResults}
In this section we present some results obtained with the methods developed in the previous sections.

First, in Section \ref{Sec:ResultsSO3Curves}, a few simple examples of curves in $\SOT$ will serve to demonstrate the basic ideas discussed in this paper: interpolation between curves and closing of open curves. Next, we will present some numerical results for the specific applications to computer animation problems: interpolation between existing motions in Section \ref{Sec:ResultsMotionInterpolation} and removing discontinuities in (almost) periodic motions in Section \ref{Sec:ResultsClosingAnimations}.

\subsection{Some implementation notes}\label{Sec:ImplementationNotes}
The motion capturing data used in the two animation application Sections has been taken from the CMU motion capture database \cite{carnegie-mellon_carnegie-mellon_2003}.\medskip

\noindent
\textbf{Computational tools for $\SOT$.}
In the examples we consider the Lie group $\SOT$ (or a product of multiple copies of $\SOT$), for which we will use a matrix representation, so left and right translations correspond to matrix products.
The Lie algebra $\mathfrak{so}(3)$ consists of the $3 \times 3$ skew-symmetric matrices which are isomorphic to vectors in $\mathbb{R}^3$ via the hat map:
\[
x =
    \left( \begin{array}{c}
        x_1\\
        x_2\\
        x_3
    \end{array} \right)
    \mapsto \hat{x} = 
    \left( \begin{array}{ccc}
        0 & -x_3 & x_2 \\
        x_3 & 0 & -x_1 \\
        -x_2 & x_1 & 0
    \end{array} \right).
\]
Efficient ways to compute the exponential map and the logarithm in $\SOT$ are available in the literature, see \cite{celledoni_lie_1999, iserles_lie-group_2000, celledoni_introduction_2014}.
In this setting, the Lie algebra exponential $\exp: \sot \rightarrow \SOT$ can be efficiently computed using Rodrigues' formula:
\[
    \exp( \hat{x} ) = I + \dfrac{ \sin ( \theta ) }{ \theta } \hat{x} + 2 \dfrac{ \sin^2 ( \theta / 2 ) }{ \theta^2 } \hat{x}^2, \quad \theta := \norm{x}.
\]
Similarly, there exist efficient means to compute the logarithm of an orthogonal matrix $X \in \SOT$:
\[
    \log ( X ) = \dfrac{\sin^{-1}( \norm{y} ) }{ \norm{y} } \hat{y}, \qquad X \neq I, \quad \text{and } X \text{ close to } I,
\] where $\hat{y} = \frac{1}{2}(X - X^{\adj})$.

As in the previous section we denote here by $X^{\adj}$ the adjoint linear operator (i.e.\ the transpose of the matrix $X$).
\smallskip

The Lie group $\SOT$ is compact and semisimple (see \cite[Lemma 2.1.4 and Example 5.5.4]{hilgert12sag}).
To compute the gradient of the error functional $\Phi$ \eqref{Eq:ErrorFunctional} as in Corollary \ref{cor: errorfun}, we will thus assume that the inner product on $\LieA$ is the negative of the Cartan-Killing form $\kappa_\LieA$.
For $\SOT$ it is well known that the Cartan-Killing form is given by 
\begin{displaymath}
 \kappa_{\sot} (X,Y) := \mathrm{tr} (XY)
\end{displaymath}
As $X,Y$ are skew symmetric, the corresponding inner product turns out to be the familiar Frobenius inner product $-\kappa_{\sot} (X,Y) = \mathrm{tr} (XY^{\adj})$.
Notice, that one can proceed similarly if $G$ is a (finite) product of copies of $\SOT$.
\medskip

\noindent
\textbf{Discrete curves.}
Given a continuous curve $c$ in $\SOT$, we approximate it by picking a discretization $\{ \theta_i \}_{i=0}^n$ of $I$ and constructing a curve $\bar{c}$ based on discrete points $\{ \bar{c}_i := c(\theta_i) \}_{i=0}^n$, between which we interpolate along geodesics in $\SOT$:
\begin{equation}\label{Eq:PiecewiseGeodesicCurve}
    \bar{c}(t) := \sum_{k=0}^{n-1} \chi_{ [ \theta_k, \theta_{k+1} ) }(t) \,\, \exp \left( \dfrac{t - \theta_k}{ \theta_{k+1} - \theta_k } \log( \bar{c}_{k+1} \bar{c}_k^{\adj}) \right) \bar{c}_k ,
\end{equation} where $\chi$ is the characteristic function.

Computing the square root velocity transform \eqref{Eq:SRVT} of such discrete curves then results in piecewise constant functions $\bar{q} = \{ \bar{q}_i \}_{i=0}^{n-1}$ in the Lie algebra, with discrete points $\bar{q}_i$ given by:
\[
    \bar{q}_i := \dfrac{\eta_i}{\sqrt{ \norm{ \eta_i} }}, \quad \eta_i := \dfrac{\log( \bar{c}_{i+1} \bar{c}_{i}^{\adj} ), 
}{\theta_{i+1} - \theta_{i}}.
\]

The inverse SRVT \eqref{Eq:SRVTInv} of a piecewise constant function $\bar{q}$ in $\LieA$, is given by a piecewise geodesic curve $\bar{c}$ in $\LieG$, as formulated in \eqref{Eq:PiecewiseGeodesicCurve}, with points $\bar{c}_i$ given by:
\[
    \bar{c}_{i+1} = \exp( \norm{ \bar{q}_i } \bar{q}_i ) \, \bar{c}_i, \qquad i = 1,\ldots,n-1, \qquad \bar{c}_0 = e.
\]

\noindent 
\textbf{Curve reparametrization.}
As we have seen, reparametrizations in $\Diffeom$ act on curves from the right.
For discrete curves, this means either a change in the underlying grid or a change in the discrete points.
For our numerical experiments, we have chosen to keep a fixed grid and resample curve points.
Applying a reparametrization $\varphi \in \Diffeom$ to the discrete curve $\bar{c}$ then results in a new discrete curve $\tilde{c}$ with sampling points $\{ \tilde{c}_i \}_{i=0}^{n}$ computed using geodesic interpolation:
\[
   \tilde{c}_i := \bar{c}_j \exp( s \log(\bar{c}_{j+1} \bar{c}_j^{\adj}) ), \qquad s := \dfrac{ \varphi(\theta_i) - \theta_j } { \theta_{j+1} -
   \theta_j }, \qquad i=0,\ldots, n,
\] where $j$ is an index such that $\theta_j \leq \varphi( \theta_i ) < \theta_{j+1}$.
Note that $\tilde{c}_0 = \bar{c}_0$ and $\tilde{c}_n = \bar{c}_n$ naturally follow from the definition of the diffeomorphism group $\Diffeom$.
\medskip

\noindent
\textbf{Curve interpolation.}
To perform interpolation between two parametrized curves $c_0$ and $c_1$, we interpolate linearly between their SRVT representations and reintegrate the result:
\begin{align}
\begin{split}\label{Eq:LinearInterpolation}
    &[0,1] \times \PSpace_* \times \PSpace_* \rightarrow C^\infty_* (I,\LieG)\\
    &(s, c_0, c_1) \mapsto \SRVTInv ( (1-s) \SRVT( c_0 ) + s \SRVT( c_1 ) ),
\end{split}
\end{align} with interpolation parameter $s$.
The construction of the curve is possible due to the vector space structure of the Lie algebra $\LieA$ and yields precisely a geodesic of $C^\infty (I, \LieA)$ (cf.\ Proposition \ref{prop: curvature}).
Notice that the interpolation map takes its image in $C^\infty_* (I,\LieG)$ and not in $\PSpace^*$, i.e.\ in general the result will not again be an immersion. 
This is due to the fact that elements in $\PSpace_*$ take their image in $\LieA \setminus \{0\}$ and this set is not convex in $\LieA$.
In practice this problem is not very serious, however it seems difficult to exclude this problem without turning to a cumbersome set of conditions on the initial data. 
\medskip

\noindent
\textbf{Curve closing.}
As we have already mentioned the Lie groups in our main example are compact and semisimple, whence Corollary \ref{cor: errorfun} is applicable.
In particular, the gradient for the curve-closing method in Equation \eqref{Eq:ClosingGradient}, takes the following form for curves in $\SOT$:
\begin{align*}
    \mathrm{grad}(\Phi)(q) &=  \| q \| c \log ( c( 1 ) )c^{\adj} + \langle c \log ( c( 1 ) )c^{\adj}, \dfrac{q}{ \| q \| } \rangle \, q
\end{align*} where $c := \SRVTInv(q)$ (computed via Lie-Euler integration), and the gradient flow can be discretized as
\[
    \bar{u}^{k+1} = \bar{u}^k - \alpha_k\, \mathrm{grad}(\Phi)(\bar{u}^k),
\] where every $\bar{u}^k$ is a discrete curve as defined above, i.e., $\bar{u}^k = \{ \bar{u}^k_i \}_{i=0}^{n}$.
Note that for the curve closing, a curve's parametrization is fixed.

This iterative approach allows us to balance accuracy and computational expense, which is useful in the computer animation applications discussed in Section \ref{Sec:ResultsClosingAnimations}.

\subsection{Curves on \texorpdfstring{$\SOT$}{the special orthogonal group}}\label{Sec:ResultsSO3Curves}
We start with two simple demonstrations of the methods developed in this article: one interpolation between curves on $\SOT$ and one application of the curve-closing algorithm laid out in Section \ref{Sec:ClosedCurves}.
In order to visualize a curve $c: I \rightarrow \SOT$, we will take one or more unit vectors in $\mathbb{R}^3$, $v_1, \ldots, v_n$, and plot the curves resulting from transforming these vectors by the successive elements in $\SOT$, i.e., we will plot $\bar{v_i}(t) := c(t) v_i$ for $i=1,\ldots, v_n$.
The curves $\bar{v_i}(t)$ will therefore evolve on the unit sphere in $\mathbb{R}^3$.

\begin{figure}[h]%[hbtp]

\begin{minipage}[b][6cm]{0.3\linewidth}
%\vspace{2cm}
\subfloat{\includegraphics[trim=7cm 0 7 1cm, clip=true]{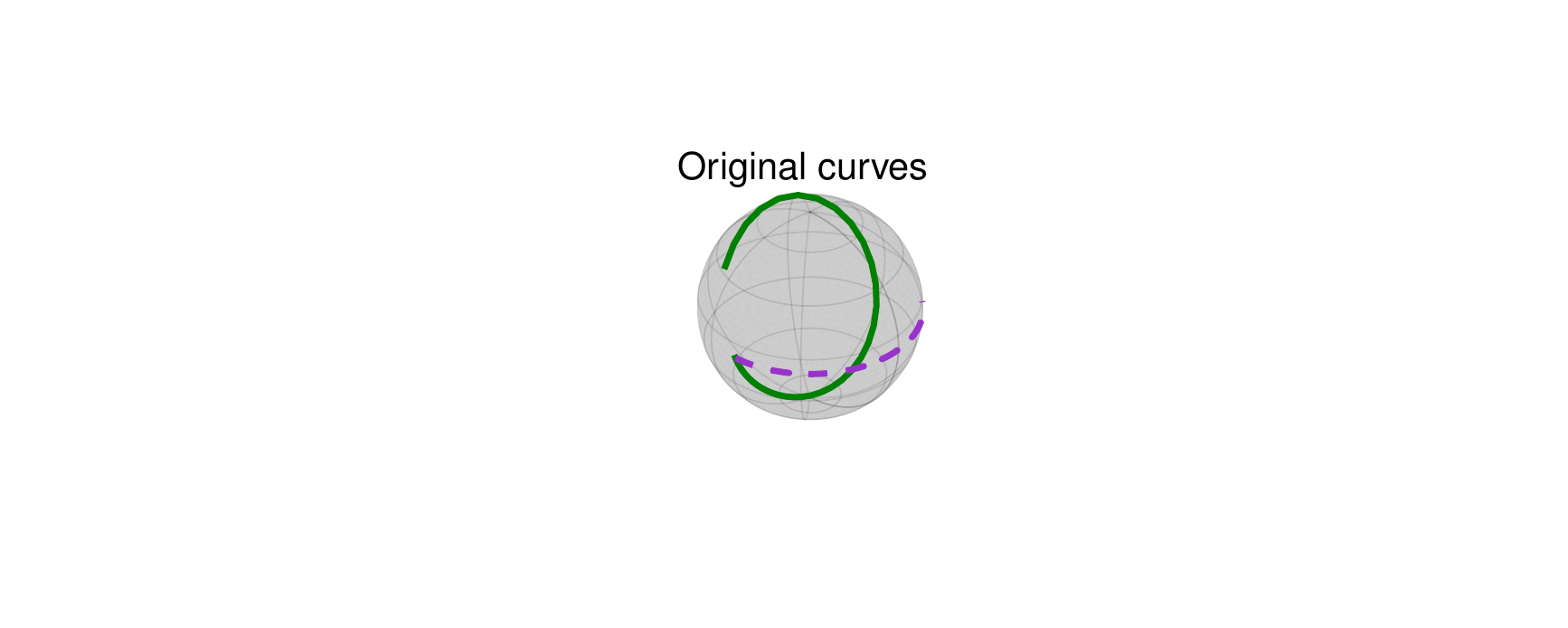}}
\end{minipage}
\begin{minipage}[b][6cm]{0.3\linewidth}
\subfloat{\includegraphics[trim=7cm 0 7 1cm, clip=true]{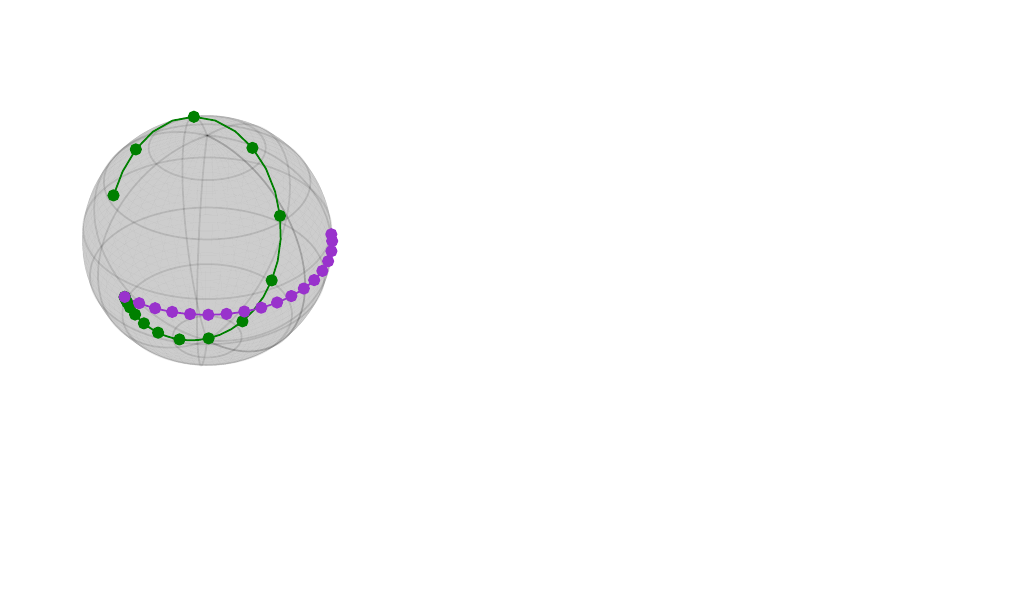}}
\end{minipage}
\begin{minipage}[b][6cm]{0.3\linewidth}
\subfloat{\includegraphics[trim=7cm 0 7 1cm, clip=true]{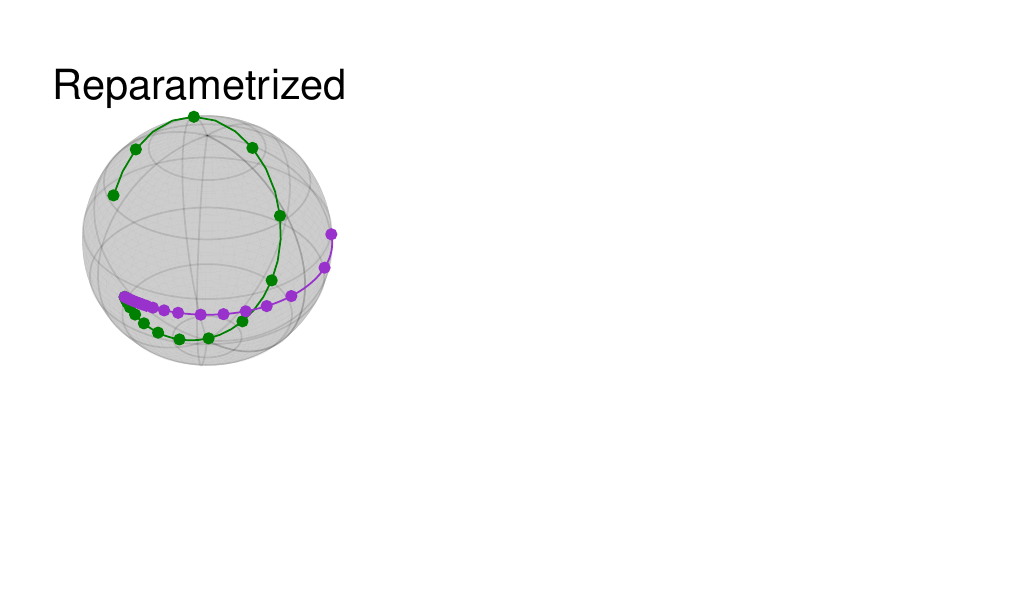}}
\end{minipage}

\vspace{-2.8cm}
\begin{minipage}[b][6cm]{0.3\linewidth}
\subfloat{\includegraphics[trim=7cm 0cm 7 1.5cm, clip=true]{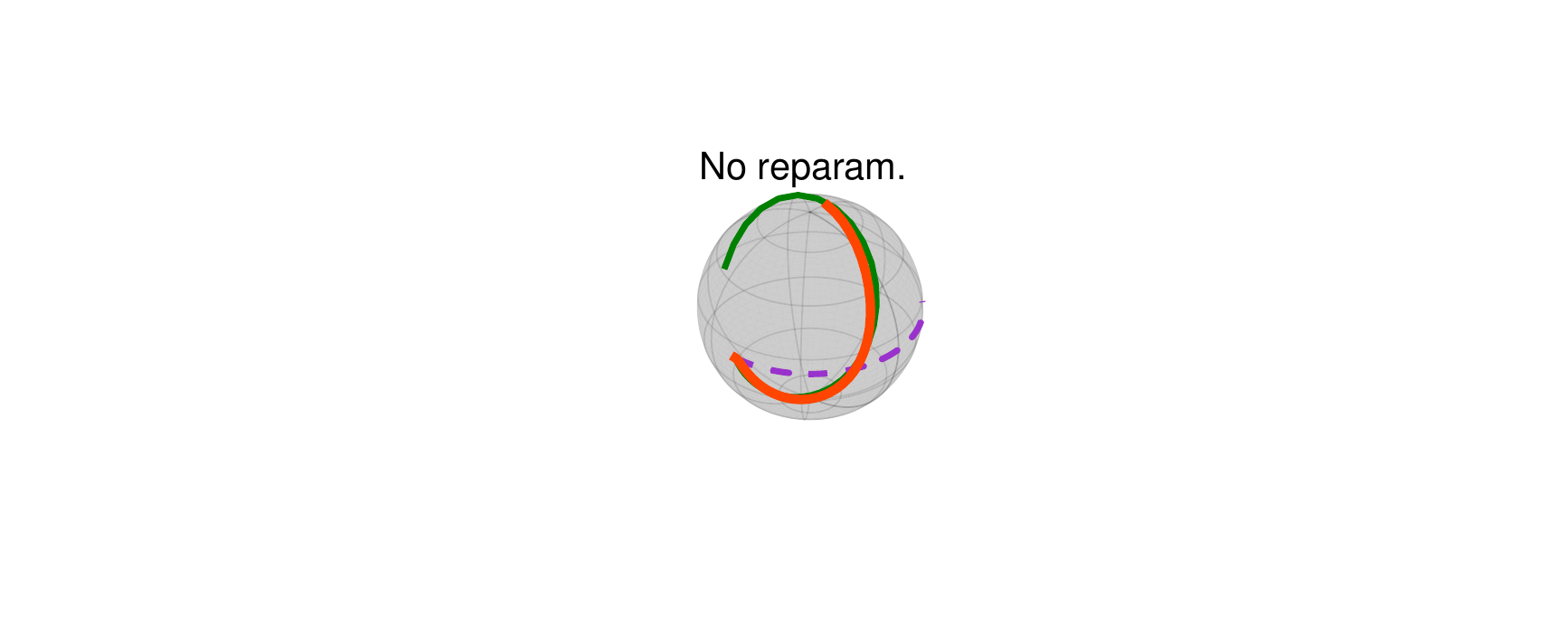}}
\end{minipage}
\begin{minipage}[b][6cm]{0.3\linewidth}
\subfloat{\includegraphics[trim=7cm 0cm 7 1.5cm, clip=true]{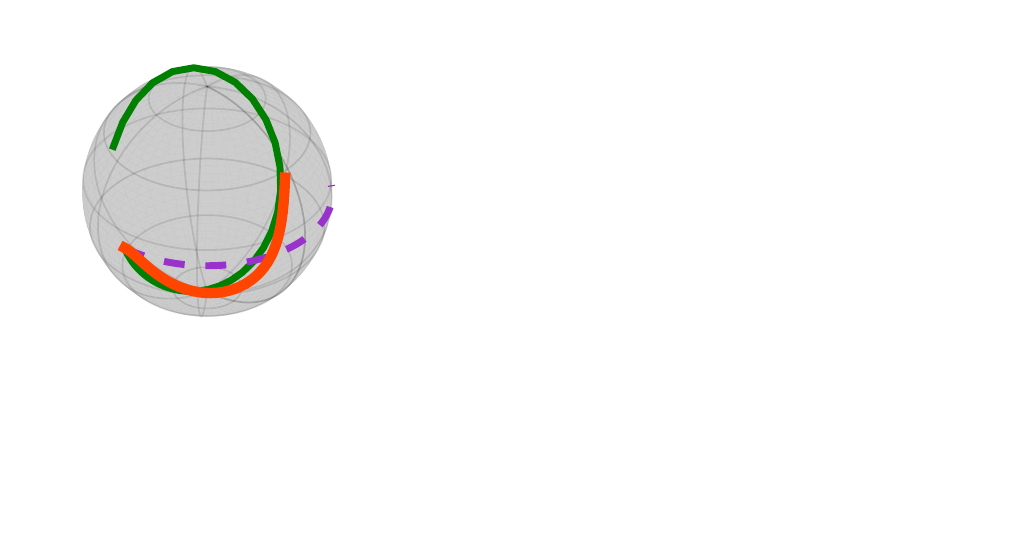}}
\end{minipage}
\begin{minipage}[b][6cm]{0.3\linewidth}
\subfloat{\includegraphics[trim=7cm 0cm 7 1.5cm, clip=true]{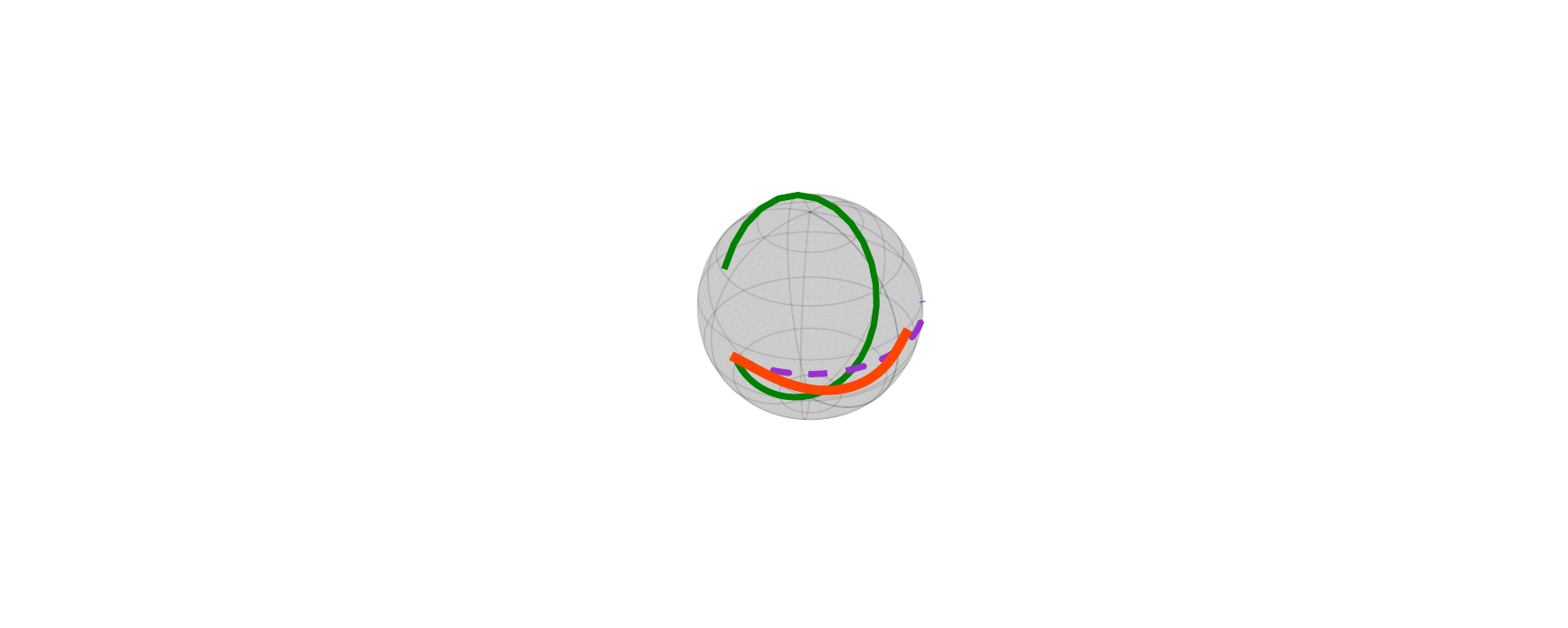}}
\end{minipage}

\vspace{-2.8cm}
\begin{minipage}[b][6cm]{0.3\linewidth}
\subfloat{\includegraphics[trim=7cm 0 7 1.5cm, clip=true]{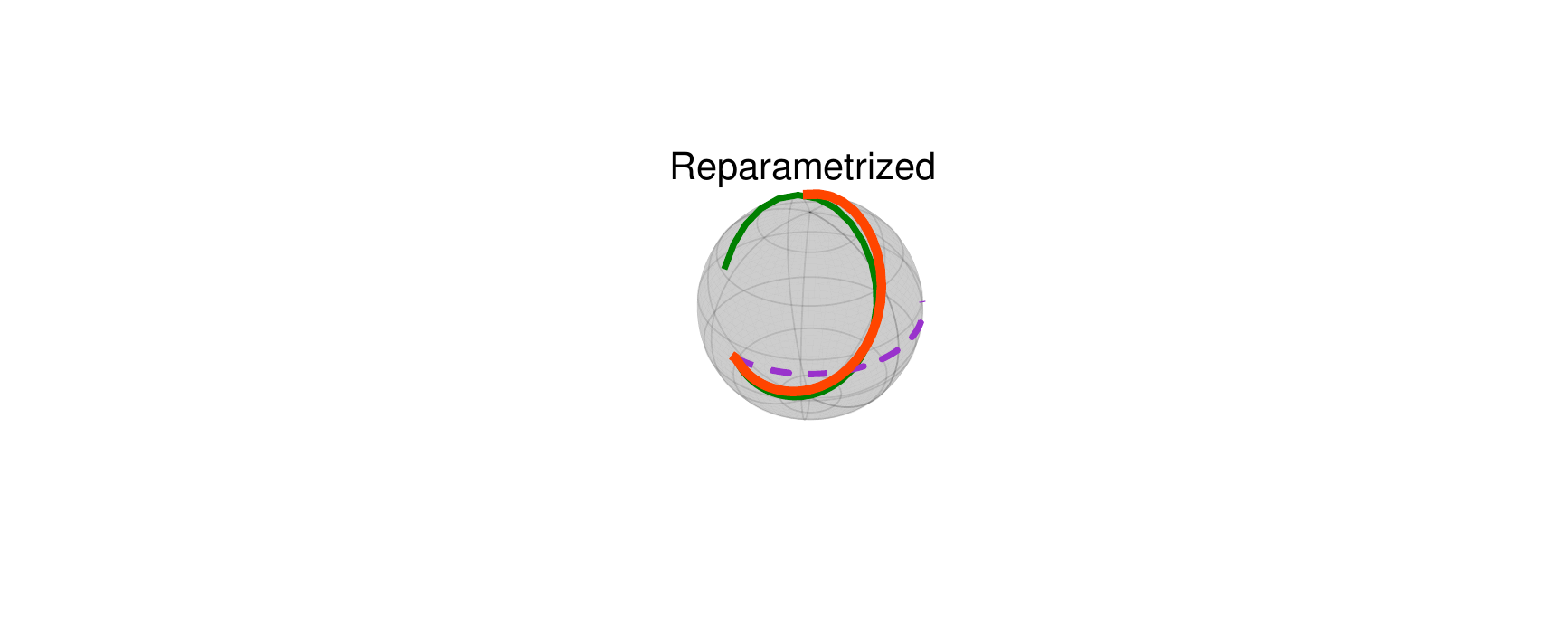}}
\end{minipage}
\begin{minipage}[b][6cm]{0.3\linewidth}
\subfloat{\includegraphics[trim=7cm 0 7 1.5cm, clip=true]{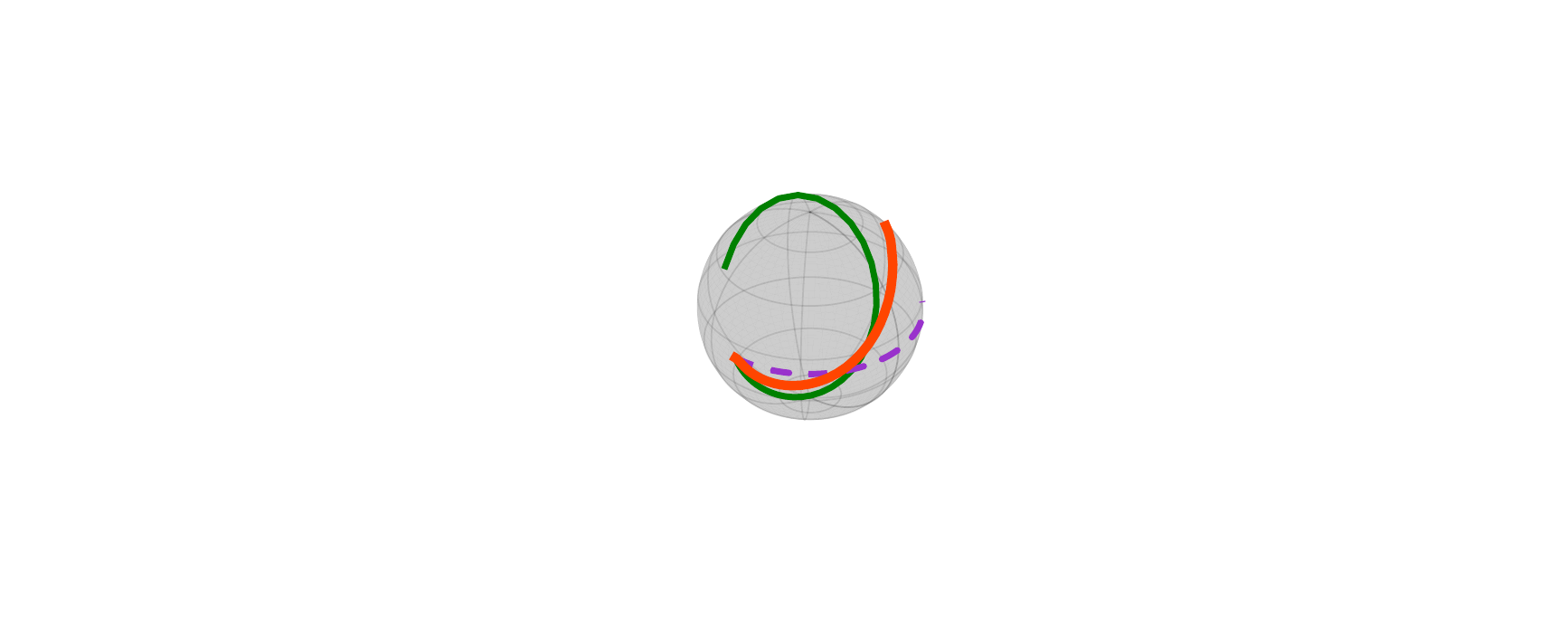}}
\end{minipage}
\begin{minipage}[b][6cm]{0.3\linewidth}
\subfloat{\includegraphics[trim=7cm 0cm 7 1.5cm, clip=true]{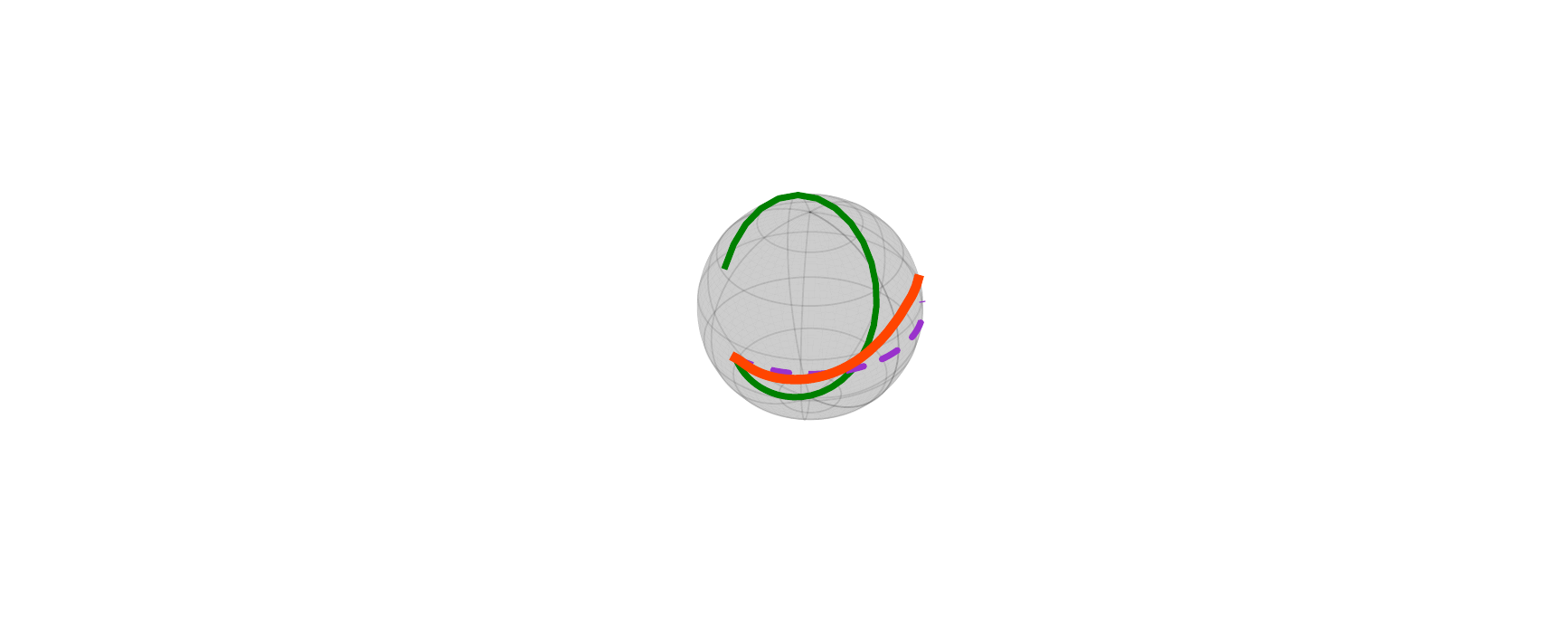}}
\end{minipage}

\vspace{-2.1cm}

\caption{\label{Fig:SO3Interpolation} Interpolation between two curves in $\SOT$ with and without reparametrization.
The curves shown here are the points traced out by the vector $(1,0,0)^T$ in $\mathbb{R}^3$ when transformed along the curves in the Lie group.
The top row shows the original curves, their parametrizations and a reparametrization performed to minimize the distance between the curves.
The thick red curves in the second and third rows are points along the geodesic path between the two original curves, at times $s \in \{1/4, 1/2, 3/4\},$ from left to right.
}
%\end{center}
\end{figure}

Figure \ref{Fig:SO3Interpolation} shows the result of interpolating between curves ${c_1, c_2: I \rightarrow \SOT}$ using the approach outlined in Sections \ref{Sec:ShapeAnalysis} and \ref{Sec:ImplementationNotes}.

This is the basic mechanism underlying the shape distance computation: We represent shapes, i.e., equivalence classes of parametrized curves under reparametrizations, by a representative, i.e., a single parametrized curve.
Then, given two shapes and two corresponding parametrized curves, we try to find a parametrization of one of the curves that minimizes the distance to the other curve.
This minimum distance is then, according to Lemma \ref{lem:RepresentationIndependence}, the distance between the two shapes, i.e., the equivalence classes.

The two curves in the first row of Figure \ref{Fig:SO3Interpolation} represent the original curves between which we interpolate using Equation \eqref{Eq:LinearInterpolation}.
In the middle figure, we see their parametrizations - whereas the orange dashed curve has a uniform parametrization, the blue one is in a sense compressed in the beginning and then stretches out.
When interpolating between those two parametrized curves, as seen in the second row, the resulting interpolation first contracts and then expands.
In the right side figure in the first row, the lower curve has been reparametrized, using dynamic programming, to better match the two curves.
This minimizes the distance between the two curves as discussed in the previous section, and the third row of figures shows the corresponding interpolating path.
\begin{center}
\begin{figure}[hbtp]
\includegraphics[scale=0.9, trim = 2cm 0 0 0, clip=false]{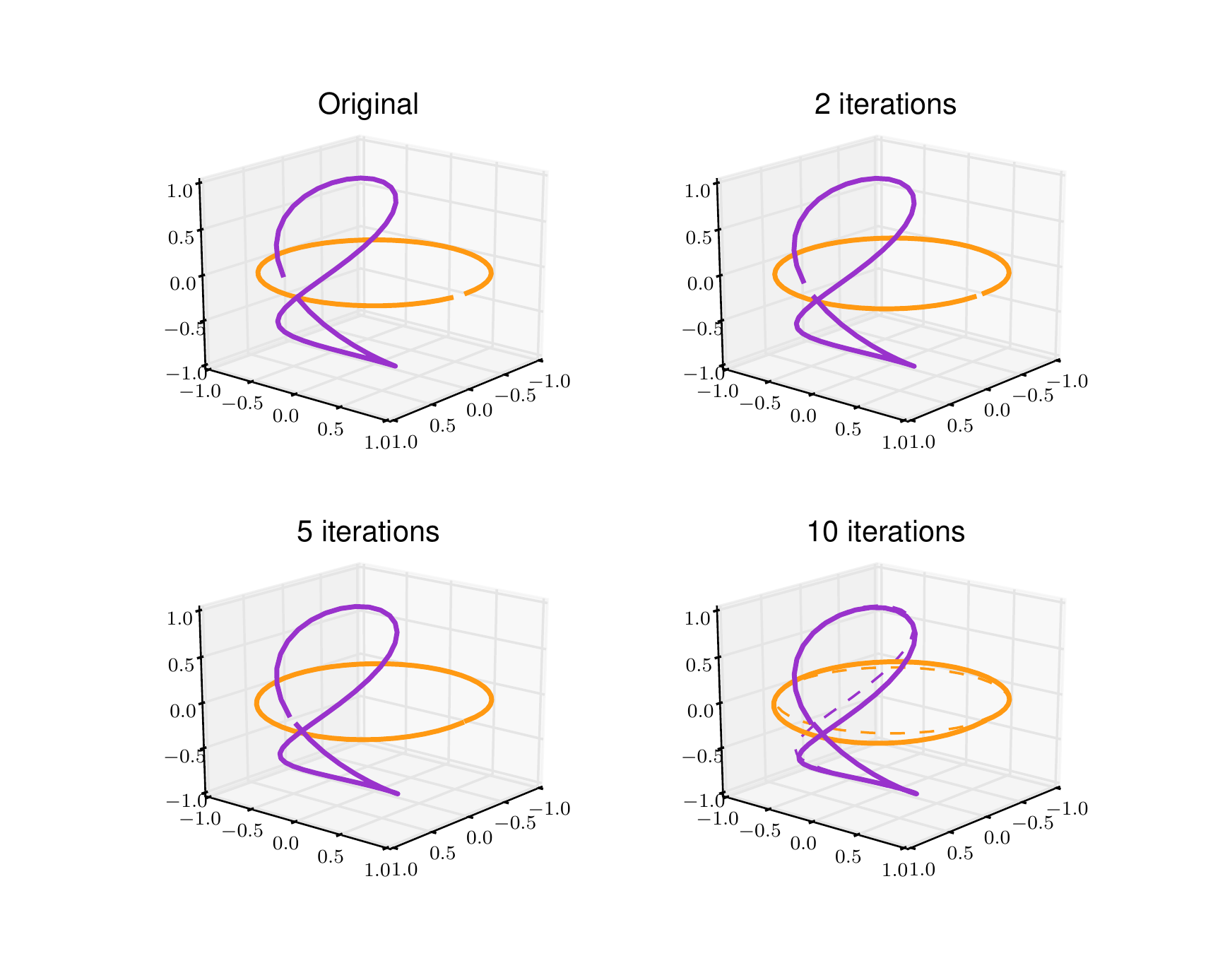}
\caption{\label{Fig:SO3Closing} Application of the closing algorithm to a single curve in $\SOT$.
The two curves shown here are the points traced out by the vectors $(1,0,0)^T$ and $(0,1,0)^T$ in $\mathbb{R}^3$ when transformed along the curve in the Lie group.
The top left figure shows the open starting curve, whereas the remaining panels show the evolution of the closing algorithm.
In the last panel, the original curves are superimposed as dashed lines to show the deviations caused by the closing method.
Note that for visualization purposes, a small stepsize was chosen, resulting in more iterations than would otherwise be necessary to achieve a satisfying accuracy.
}
\end{figure}
\end{center}
\afterpage{\clearpage}
In Figure \ref{Fig:SO3Closing}, we see the results of applying the curve closing algorithm to an open curve in $\SOT$.
We have plotted the results at different stages of the iterative algorithm to highlight the evolution of the closing process, i.e., how the curves iteratively move towards closedness.
In this toy-example, we have chosen a small stepsize to better show the behaviour of the algorithm.
In practice, just two to three iteration steps are typically enough to achieve a sufficient degree of closedness in the curve.

\subsection{Motion interpolation}\label{Sec:ResultsMotionInterpolation}

We now want to apply the interpolation method outlined earlier and demonstrated in the previous Section to entire motions of virtual characters.
This means that instead of curves in $\SOT$, we are now dealing with curves in $\SOT^d$, with a copy of $\SOT$ for every joint in the animated character.
The numerical methods as described in Section \ref{Sec:ImplementationNotes} stay the same however.

\begin{figure}[t]
\begin{center}
\includegraphics[scale=1, trim = 3cm 0 0 0, clip=false]{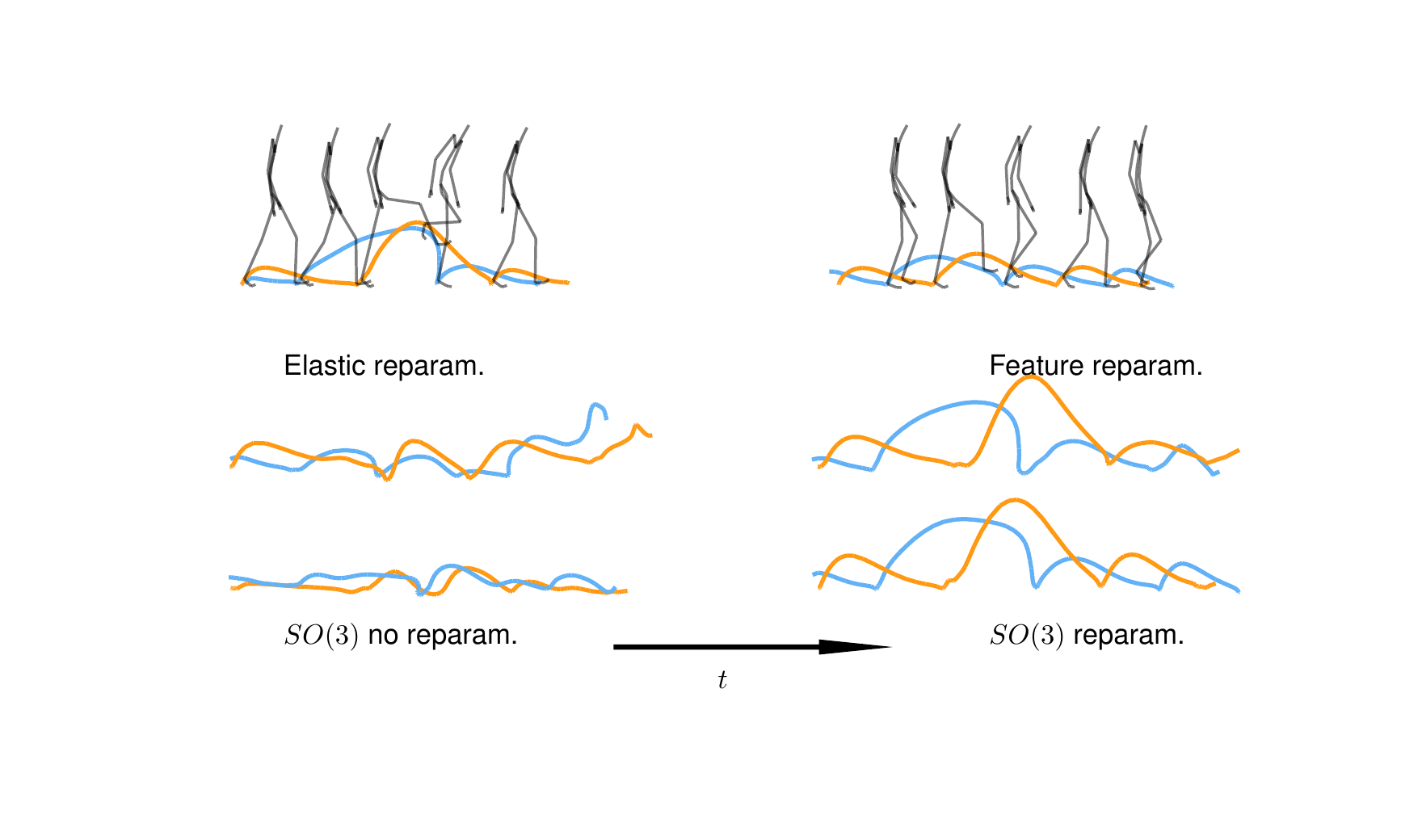}
\vspace{-2cm}
\caption{\label{Fig:ObstacleInterpolation} Interpolation between two walking animations, shown in the top row.
The two lines shown in the plots are the trajectories of both feet, with the big steps over the obstacles clearly visible.
Below, interpolations between the two animations computed using different methods are plotted.
The upper row shows existing results from \cite{bauer_landmark-guided_2015}, where additional feature point information was used to guide the interpolation.
The bottom row shows interpolation results for matching the $\SOT^d$ animations without (on the left) and with (on the right) reparametrization.
No additional feature point information was necessary for this method.
(Note that the feature matching approach by design requires reparametrization.)}
\end{center}
\end{figure}

In Figure \ref{Fig:ObstacleInterpolation} we start with two original motions: One with the animated character stepping over a high obstacle and one with the character stepping over a low obstacle.
We want to interpolate between those two motions, allowing us to let the character step over obstacles of arbitrary heights.
This problem was already considered in \cite{bauer_landmark-guided_2015}, where it was found that the approach of parametrizing animations with Euler angles and then performing shape analysis \cite{eslitzbichler_modelling_2014} produced unsatisfactory results for these motions and had to be augmented with extra landmark information to achieve a realistic interpolation result.
Here we find that, for this problem, the use of a Lie group formulation allows us to achieve similarly good results as in \cite{bauer_landmark-guided_2015}, but without the need for additional landmark information.
Figure \ref{Fig:ObstacleInterpolation} shows results for the Lie group formulation once without and once with reparametrizations.
Reparametrizations in the context of animations, also known as \emph{time warps} \cite{kovar_flexible_2003}, serve to align two animations on the time axis, i.e., they can speed up or slow down parts of an animation to match the motions more closely.
We see that in this case, reparametrizations are necessary to achieve satisfying blending results.
This is likely due to a slight phase shift in the beginning (compare the starting points of the trajectories on both original animations in Figure \ref{Fig:ObstacleInterpolation}).\\

A natural extension of this scheme would be to interpolate between multiple animations.
Among others, this could be useful for extended walking animations where combining multiple animations and varying the interpolation weights over time could help produce a large but consistent set of motions to avoid repetitive visual elements.
This could for example be accomplished by formulating this animation blending problem as a Karcher mean (\cite{srivastava_shape_2011, su_statistical_2014}) of multiple animation curves.
We will consider such a scheme in future work.

\subsection{Periodic motions}\label{Sec:ResultsClosingAnimations}

\begin{figure}
\begin{center}
%\vspace{-3cm}
\includegraphics[scale=1, trim = 2.5cm 0 0 0, clip=false]{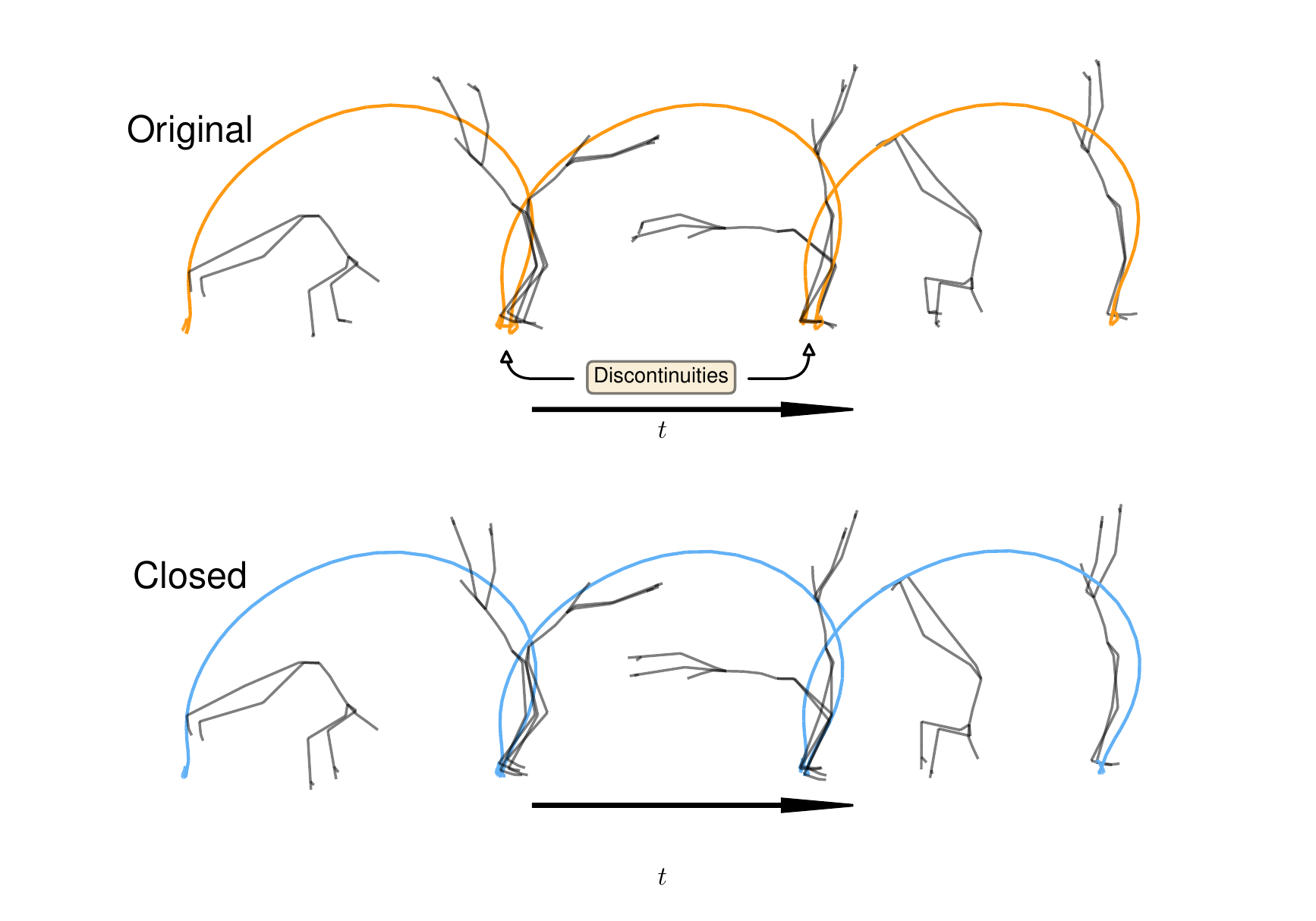}
%\vspace{-1cm}
\caption{\label{Fig:ClosedAnimation1} Application of closing algorithm to a handspring animation.
The motion is repeated twice. Note how, when repeating the original animation, the feet jump backwards to cover a gap between the first and last frame in the animation.
In the bottom plot, the curve closing method has been used to remove this discontinuity.
%We refer the reader to the supplementary material for an video of this animation.
}
\end{center}
\end{figure}

As a second practical application, we will now look at two examples of how the curve closing algorithm can be applied to motion data to create periodic animations.
A similar scheme was developed in \cite{eslitzbichler_modelling_2014}, but based on an Euler angle parametrization for the character joints.
For animations exhibiting a large range of motions such as rolls and flips, this can result in highly degenerate results with characters seemingly just floating in the air.
The Lie group based algorithm developed in Section \ref{Sec:ClosedCurves} exhibits much stronger stability to such outliers.

In Figure \ref{Fig:ClosedAnimation1}, we start with a handspring motion that we repeat three times.
The plotted curve in the figure shows the trajectory of the right foot of the character.
Following this trajectory in the upper part of the figure, we can see a discontinuity in the foot position when the animation repeats.
The problem is that the start and end poses of the animation are too different, which results in a noticeable jerk when we repeat the handspring.
Note that while this gap may seem small in the static picture, it is much more noticeable when looking at the actual animation.
Figure \ref{Fig:ClosedAnimation1Overlap} shows the discontinuity in more detail.
%We refer the reader to the supplementary material to see a video of this.

\begin{figure}[hbtp]
\vspace{1cm}
\begin{minipage}[b][6cm]{0.47\linewidth}
\subfloat{\includegraphics[trim=1cm 0 0 0cm, clip=true]{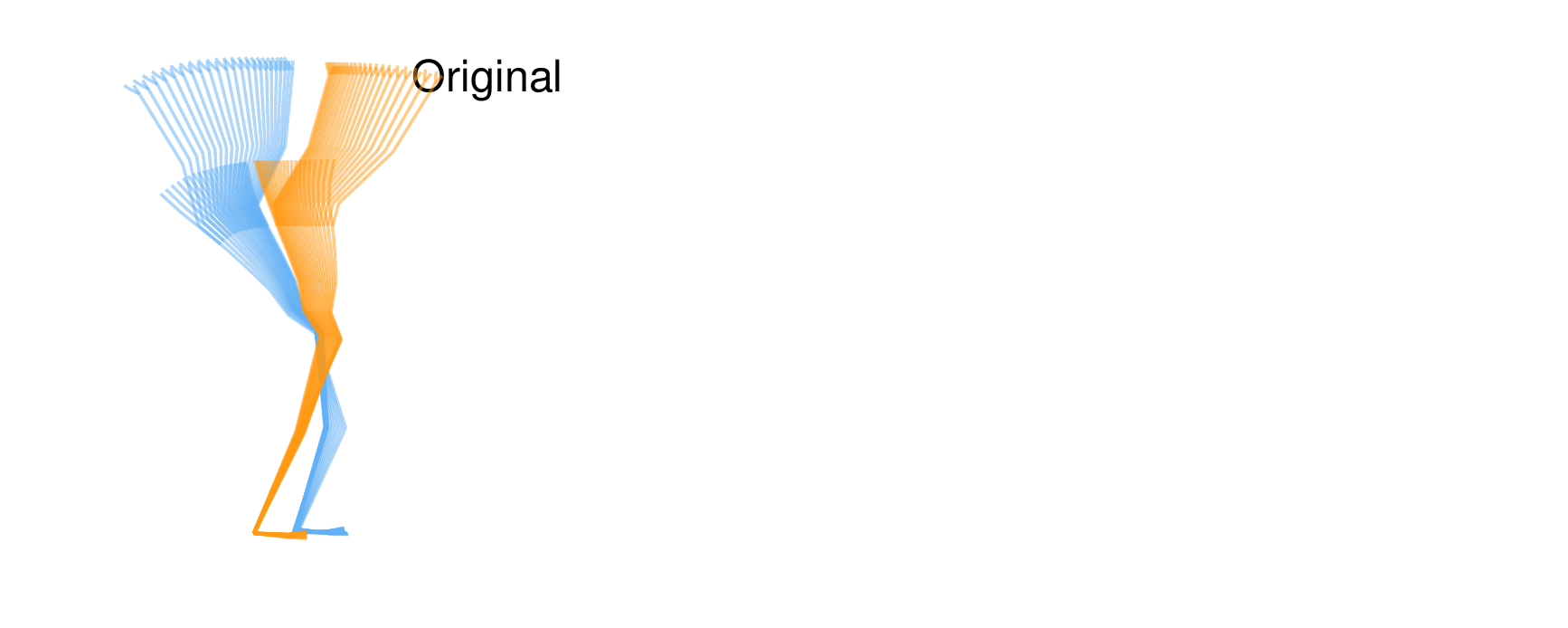}}
\end{minipage}
\begin{minipage}[b][6cm]{0.47\linewidth}
\subfloat{\includegraphics[trim=9cm 0 0 0cm, clip=true]{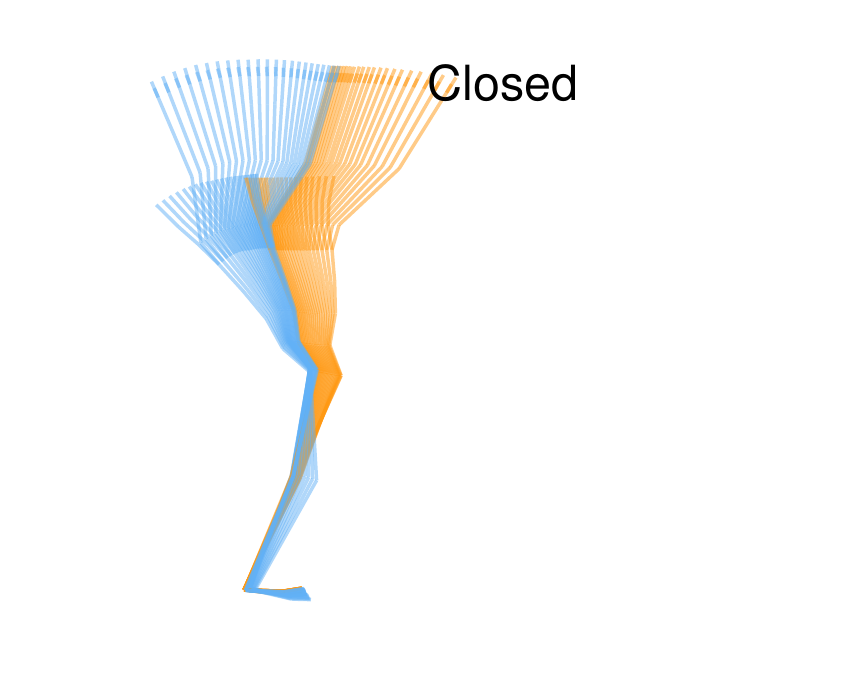}}
\end{minipage}
\vspace{-1cm}
\caption{\label{Fig:ClosedAnimation1Overlap} A closer look at the discontinuities in the handspring animation of Figure \ref{Fig:ClosedAnimation1}.
The (blue) left half of both figures shows the last few animation sampling points of the animation, the (orange) right half shows the first few sampling points when the animation is repeated. 
To avoid cluttering the figure, only the left half of the skeleton has been plotted.
}
\end{figure}

\begin{figure}[t]
\vspace{-1.5cm}
\begin{center}
\includegraphics[scale=1, trim = 2.5cm 0cm 0 0cm, clip=true]{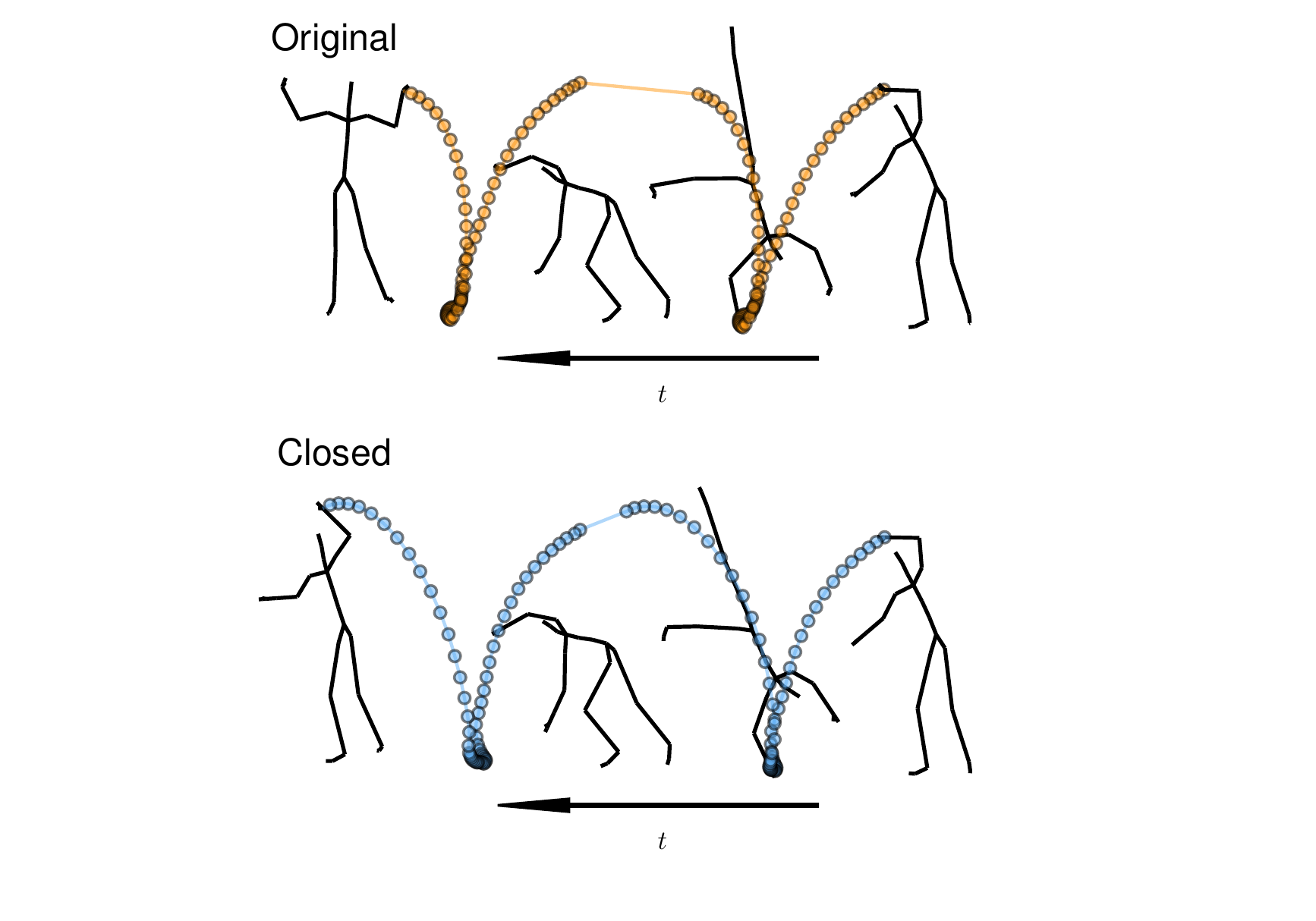}
\vspace{-1.5cm}
\caption{\label{Fig:ClosedAnimation2} Application of closing algorithm to a cartwheel animation.
Note that large different between start and end poses, on the right and the left respectively.
The motion is repeated once and suffers from a strong jerk when it repeats, especially in the left hand.
In the second row, the curve closing method has been used to alleviate this discontinuity.
The remaining gap could be closed further using additional iterations of the closing algorithm, but this eventually introduces visual artifacts, with the left hand sliding on the ground during the cartwheel.}
\end{center}
\end{figure}

On the bottom of Figure \ref{Fig:ClosedAnimation1}, we see the result of applying a few iterations of the curve closing method to the handspring animation.
The discontinuity has been strongly reduced, while the rest of the motion has been preserved.
The result is an aesthetically much more pleasing motion.

Figure \ref{Fig:ClosedAnimation2} shows another example of the animation closing method.
In this case, the character performs a cartwheel, and the animation starts and ends in two very different poses, which causes a big discontinuity, particularly in the left hand, when the animation is repeated.
The closing algorithm manages to close this gap to a large extent, while preserving the general appearance of the animation.
Note that the gap could be closed further by running more iterations of the closing algorithm, but at the expense of introducing visual artifacts, such as the left hand sliding on the ground during the cartwheel.

\section{Conclusion}
In this article, we have formulated a shape analysis framework for curves on Lie groups based on the SRVT approach \cite{srivastava_shape_2011}.
This has allowed us to construct efficient algorithms to solve two very different problems in computer animation: Interpolating between animations and generating cyclic animations.
Potential further applications include classification and search of animations.

Future work in this area could involve both joint-wise constraints (e.g., knees are not allowed to bend backwards) as well as blending of multiple animations using a Karcher mean approach, as described in Section \ref{Sec:ResultsMotionInterpolation}.
Also, investigating higher order continuity in the curve closing algorithm could lead to improved practical results. 

\section*{Acknowledgements}
This work has received funding from the European Unions Horizon 2020 research and innovation programme under the Marie Sk\l{}odowska-Curie grant agreement No.\ 691070.
The authors would like to thank Markus Grasmair and Rafael Dahmen for valuable discussions.
The data used in this project was obtained from mocap.cs.cmu.edu. The database was created with funding from NSF EIA-0196217.
Finally, we thank the anonymous referees whose insightful comments helped to improve this paper.

\begin{appendix}
 \section{Detailed proofs for Section 3}
 
 \subsection{Proof of Proposition \ref{Prop: trlog}}\label{App: linearLie}
 
 In Proposition \ref{Prop: trlog} a formula for the tangent map of the right-logarithmic derivative was given. 
 For the readers convenience we repeat the formulation of the proposition now.
% Since it is much easier to understand let us provide a proof first for the special case of $\LieG$ being a linear Lie group.
\medskip
 
\textbf{Proposition}\emph{
 Let $c \in C^\infty (I,\LieG)$ and $v \in T_c C^\infty (I,\LieG) = \{w \in C^\infty (I,TG) \mid w(t) \in T_{c(t)} G,\ \forall t \in I\}$. 
 Then writing $\left[ \cdot, \cdot \right]$ for the Lie bracket in $\LieA$ we have}
  \begin{displaymath}
   T_c \delta^r (v) (t) = \frac{\dif}{\dif t} \left(R_{(c(t))^*}^{-1}(v(t)) \right) + \left[R_{(c(t))^*}^{-1}(v(t)), \delta^r (c) (t) \right].
  \end{displaymath}
Finally, let us write out the above formula for $T_c\delta^r$ in the case that $\LieG$ is a linear Lie group:

\begin{remark}[Proposition \ref{Prop: trlog} for linear Lie groups]
 Let $\LieG$ be a (finite-dimensional) linear Lie group, $c \colon I \rightarrow \LieG$ smooth and $v \in T_c C^\infty (I,\LieG)$, i.e.\ $v\colon I \rightarrow T\LieG$ with $v(t) \in T_{c(t)}\LieG$. 
 Then $$T_c\delta^r (v)(t) = \frac{\dif}{\dif t} (v(t) \cdot c(t)) + \left[ v(t)\cdot c^{-1}(t) , \dot{c}(t) \cdot c^{-1}(t)\right], $$
 where products are matrix products and the Lie bracket is the commutator bracket.  
\end{remark}
\begin{proof}[Proof of Proposition \ref{Prop: trlog}]
 Recall from \cite[38.1 Lemma]{kriegl97tcs} the product rule for the logarithmic derivative:
\begin{displaymath}
 \delta^r (f\cdot g)(t) = \delta^r (f)(t) + \mathrm{Ad} (f (t)) (\delta^r (g)(t)) \quad \text{ for } f,g \colon I \rightarrow G \text{ smooth}
\end{displaymath}
Here $\mathrm{Ad}(g) = (L_g)_* \circ (R_{g^{-1}})_*$ is the adjoint action of $\LieG$ on $\LieA$.
Observe that the derivative of $\mathrm{Ad}$ in the unit $e \in G$ is the derived representation of $\LieG$ on $\LieA$, which is $T_e \mathrm{Ad} (x)(y) = [x,y]$, i.e.\ the Lie bracket in $\LieA$.
Moreover, recall from \cite[Lemma 2.1]{glockner_regularity_2012} that $T_{e_\LieG} \delta^r (\xi)(t) = \tfrac{\dif}{\dif t}\xi(t)$, where $e_G \colon I \rightarrow \LieG$ is the unit in $C^\infty (I,\LieG)$ (cf.\ Proposition \ref{prop: LGP:prop}).
Now fix a curve $\gamma \colon ]-a,a[ \times I \rightarrow G$ with $\gamma(0,t) = c(t)$ for all $t \in I$ and $\left.\frac{\partial}{\partial \varepsilon}\right|_{\varepsilon =0} \gamma (\varepsilon ,t ) = v(t)$. 
Then the tangent map of $\delta^r$ can be computed as follows:
\begin{equation} \begin{aligned} \label{eq: derivative}
  T_c \delta^r (v) (t) &= \left.\frac{\partial}{\partial \varepsilon}\right|_{\varepsilon =0}\delta^r(\gamma (\varepsilon , t)) = \left.\frac{\partial}{\partial \varepsilon}\right|_{\varepsilon =0}\delta^r(\gamma (\varepsilon , t) \cdot (c(t))^{-1} \cdot c(t))\\
		       &= \left.\frac{\partial}{\partial \varepsilon}\right|_{\varepsilon =0}\left(\delta^r(\gamma (\varepsilon , t) \cdot (c(t))^{-1}) + \mathrm{Ad} (\gamma (\varepsilon , t) \cdot (c(t))^{-1}) (\delta^r c (t)) \right)\\
		       &= T_{e_G}\delta^r \left(\left.\frac{\partial}{\partial \varepsilon}\right|_{\varepsilon =0}\left(\gamma (\varepsilon , t) \cdot (c(t))^{-1}\right) \right)\\ 
		       & \hspace{1cm} + T_e \mathrm{Ad} \left(\left.\frac{\partial}{\partial \varepsilon}\right|_{\varepsilon =0} \left(\gamma (\varepsilon , t) \cdot (c(t))^{-1}\right)\right) (\delta^r c (t))\\
		       &= \frac{\partial}{\partial t} \left(R_{(c(t))^*}^{-1}(v(t)) \right) + \left[R_{(c(t))^*}^{-1}(v(t)), \delta^r (c) (t) \right] \qedhere
   \end{aligned}
  \end{equation}
\end{proof}

\subsection{Proof of Theorem \ref{thm: globaldist}}\label{App: distance}
 The aim of this section is to prove Theorem \ref{thm: globaldist}.
 Before we do this, we need the following auxiliary result.
 
 \begin{lemma}\label{lem: projector}
  Let $H$ be a Hilbert space with \textup{dim} $H>2$ and $S_H := \{ v \in H \mid \norm{v} = 1\}$ its unit sphere.
  For $i=1,2$ fix smooth curves $u_i \colon I \rightarrow S_H$ on a compact interval $I$.\footnote{Recall that the unit sphere of a Hilbert space is a closed submanifold by \cite[p. 29 Example]{MR1666820}, whence it makes sense to consider smooth curves to $S_H$.}
  Then there is an open subset $O \subseteq S_{H}$ which contains the image of $u_1$ and $u_2$ together with a diffeomorphism $\Phi \colon O \rightarrow V$ to a vector space $V$ with \textup{dim} $V >1$.
 \end{lemma}
 
 \begin{proof}
 We have to distinguish two cases: 
 
 \textbf{Case 1:} $2 < $dim $H < \infty$. 
 Since $I$ is an interval and $\text{dim } S_{H} = d >1$, Sards Theorem (see \cite[XVI, \S 1 Theorem 1.4]{MR1666820}) implies that the smooth curves $u_1, u_2 \colon I \rightarrow S_{H}$ are not surjective.
 Hence we can pick $x \in S_{H}$ such that $x \not \in u_1 (I) \cup u_2 (I)$.
 Set $O= S_{H} \setminus \{x\}$ and let $\Phi \colon O \rightarrow \mathbb{R}^d$ be the stereographic projection through $x$. Then $\Phi$ is a diffeomorphism as needed.
 
 \textbf{Case 2:} \emph{$H$ is infinite-dimensional.} 
 As $H$ is an infinite-dimensional Hilbert space, it is well known (cf.\ \cite{MR1363804}) that $S_{H}$ is diffeomorphic to $H$ itself. 
 Hence set $O = S_{H}$ and let $\Phi \colon O \rightarrow H$ be the diffeomorphism constructed in ibid.
\end{proof}
 Notice that the proof of Theorem \ref{thm: globaldist} below does \textbf{not} generalise to the case $\mathrm{dim}\ \LieA =2$.\medskip

\textbf{Theorem} \emph{ If $\mathrm{dim}\ \LieA > 2$ then the geodesic distance of $C^\infty (I,\LieA\setminus \{0\})$ is globally given by the $L^2$-distance.  }
  %In particular, in this case the geodesic distance of the pullback metric \eqref{Eq:ElasticMetric} on $\PSpace_*$ is given by the distance function \eqref{Eq:L2Distance}.}

   \begin{proof}[Proof of Theorem \ref{thm: globaldist}]
 Let $q_1 , q_2 \in C^\infty (I, \LieA \setminus \{0\})$. 
 Denote by $d_{C^\infty (I,\LieA \setminus \{0\})}$ the geodesic distance in $C^\infty (I,\LieA \setminus \{0\})$. 
 Observe first that the geodesic distance coincides locally with the $L^2$-distance $d_{L^2}$. 
 In particular, this implies $$d_{C^\infty (I,\LieA \setminus \{0\})} (q_1,q_2) \geq d_{L^2} (q_1,q_2).$$
 
 We have to make sure that both distances coincide even if the minimizing geodesic $c_{q_1,q_2} (s) = (1-s)q_1 + s q_2$ in $C^\infty (I,\LieA)$ is not contained in $C^\infty (I,\LieA \setminus \{0\})$.
 Obviously, the minimizing geodesic is not contained in $C^\infty (I,\LieA \setminus \{0\})$ if there are $s,t$ such that $c_{q_1,q_2} (s)(t) = 0$.
 Such $(s,t)$ can exist if and only if $\frac{q_1 (t)}{\norm{q_1 (t)}} = - \frac{q_2 (t)}{\norm{q_2 (t)}}$.
  Our aim is now to find a smooth perturbation of $u_1$ which is arbitrarily close to $u_1$ (with respect to the $L^2$-norm) such that the linear paths $c_{v_1 ,q_1} (s) = sv_1 + (1-s)q_1$ and $c_{v_1,q_2} (s) = sv_1 + (1-s)q_2$ are contained in $C^\infty (I, \LieA \setminus \{0\})$. 
 We split the problem into two distinct steps:
 \medskip
 
 \textbf{Step 1:} \emph{Construct a perturbation $v_1$ of $q_1$ such that $c_{v_1,q_2}$ is in $C^\infty (I,\LieA \setminus \{0\})$.} 
 Set $u_1 := \frac{q_1}{\norm{q_1}}$ and $u_2 := -\frac{q_2}{\norm{q_2}}$ to obtain maps which take their image in the unit sphere $S_{\LieA} = \{ v \in \LieA \mid \norm{v}=1\}$.
 Since $\LieA$ is a Hilbert space the unit sphere $S_{\LieA}$ is a closed submanifold of $\LieA$ (cf.\ \cite[p. 29 Example]{MR1666820}) and we see that $u_i \in C^\infty (I,S_{\LieA})$ for $i=1,2$.
 Now it suffices to construct a smooth perturbation $\tilde{v} \colon I \rightarrow S_{\LieA}$ of $u_1$ such that $\tilde{v} (t) \neq u_2 (t)$ for all $t \in I$.
 By Lemma \ref{lem: projector}, there is an open set $O \subseteq S_{\LieA}$ which contains the images of $u_1$ and $u_2$ together with a diffeomorphism $\Phi \colon O \rightarrow V$ to some vector space $V$ with $\text{dim} V > 1$.
 Consider now the smooth curve $w_1 = \Phi \circ u_1 - \Phi \circ u_2  \colon I \rightarrow V$. Since $\text{dim } V > 1$ we can clearly construct a smooth mapping $\tilde{w} \colon I \rightarrow V\setminus \{0\}$ which satisfies 
 \begin{equation}\label{eq: est1}
 \sup_{t\in I} \norm{\Phi \circ u_1 - (\tilde{w} + \Phi \circ u_2)} = \sup_{t \in I}\norm{w_1 (t) - \tilde{w} (t)} < \delta
 \end{equation}
 for some arbitrary but fixed $\delta >0$ (the control $\delta$ will be needed in Step 2 below).
 Then $\tilde{v} := \Phi^{-1} \circ (\tilde{w} + \Phi \circ u_2) \colon I \rightarrow S_{\LieA}$ is smooth, and satisfies $\tilde{v} (t) \neq u_2 (t)$ for all $t \in I$ (since $\Phi (\tilde{v}) = \tilde{w}(t) + \Phi \circ u_2(t) \neq \Phi \circ u_2 (t)$). 
 Define $v_1(t) := \tilde{v}(t)\cdot \norm{q_1}(t)$ to obtain a smooth map such that $c_{v_1,q_2} (s) = sv_1 + (1-s)u_2$ is a path in $C^\infty (I,\LieA \setminus \{0\})$.
  \medskip
 
 \textbf{Step 2:} \emph{Adjust $v_1$ such that also $c_{v_1,q_1}$ is in $C^\infty (I,\LieA \setminus \{0\})$.}
 Since $q_1 \colon I \rightarrow \LieA \setminus \{0\}$ is smooth and $I$ is compact, $r : = \inf_{t\in I} \norm{q_1 (t)} >0$.
 Hence if \begin{equation}\label{eq: bound}
           \sup_{t \in I}\norm{q_1 (t)- v_1 (t)} < r,
          \end{equation}
 then $c_{v_1,q_1}$ is contained in $C^\infty (I,\LieA \setminus \{0\})$. 
 To see that we can choose $v_1 =  \norm{q_1} \cdot \Phi^{-1} (\tilde{w} + \Phi \circ u_2)$ such that \eqref{eq: bound} is satisfied, we consider the map  
  \begin{displaymath}
   \theta \colon C (I, V) \rightarrow C (I,\LieA \setminus \{0\}),\quad f \mapsto (t \mapsto \norm{q_1(t)} (\Phi^{-1} (f(t) + \Phi \circ u_2 (t))))
  \end{displaymath}
  and from Step 1 we have $v_1=\theta(\tilde{w})$. 
  We claim that $\theta$ is continuous if we endow the spaces of continuous maps $C (I, V)$ and $C (I,\LieA \setminus \{0\})$ with the topology induced by $\norm{f}_\infty : = \sup_{t \in I} \norm{f(t)}$ (where $\norm{\cdot}$ denotes the norm of $V$ and $\LieA$, respectively).
  If this is true, then the proof can be completed as follows. 
  Since $\theta$ is continuous with $\theta (\Psi \circ \frac{q_1}{\norm{q_1} - \Psi \circ u_2})=q_1$, we can choose $\delta_r >0$ so small to ensure that \eqref{eq: bound} holds if \eqref{eq: est1} is satisfied.
  Thus $c_{v_1,q_1}$ takes its image in $C^\infty (I,\LieA \setminus \{0\})$ if $\delta_r >0$ is small enough.
 \medskip
 
 Summing up, we have seen that we can always construct a smooth perturbation $v_1$ of $q_1$ such that the linear paths $c_{v_1,q_1}$ and $c_{v_1,q_2}$ are contained in $C^\infty (I,\LieA \setminus \{0\})$.
 Moreover, for $\varepsilon >0$ the estimates in Step 2 show that we can choose $\delta_r > \delta_\epsilon>0$ such that the left hand side of \eqref{eq: bound} is smaller than $\varepsilon$.
 Hence for each $\varepsilon >0$ we can choose a smooth map $v_1^\varepsilon$ such that $d_{L^2} (v_1^\varepsilon, q_1) < \varepsilon$. 
 In particular, $v_1^\varepsilon$ converges to $q_1$ with respect to the $L^2$-distance.
 Then the geodesic distance satisfies  
 \begin{align*}
   d_{C^\infty (I,\LieA \setminus \{0\})} (q_1,q_2) &\leq d_{C^\infty (I,\LieA \setminus \{0\})} (q_1, v_1^\varepsilon) + d_{C^\infty (I,\LieA \setminus \{0\})} (v_1^\varepsilon , q_2) \\
						    &=d_{L^2}(q_1,v_1^\varepsilon) + d_{L^2} (v_1^\varepsilon, q_2) \\
						    &\leq \varepsilon + d_{L^2} (v_1^\varepsilon, q_2)  \xrightarrow{\varepsilon \rightarrow 0} d_{L^2}(q_1,q_2).
 \end{align*}
 Thus the geodesic distance of $C^\infty (I,\LieA\setminus \{0\})$ coincides with the $L^2$-distance.\medskip
 
 \textbf{Proof of the claim: $\theta$ is continuous.} Recall that the topology on $C (I,V)$ induced by $\norm{\cdot}_\infty$ coincides with the compact open topology.
 Hence \cite[Theorem 3.4.2]{MR1039321} shows that the map $\gamma_\Phi \colon C(I,V) \rightarrow C(I,S_{\LieA}), f \mapsto \Phi^{-1} \circ f$ is continuous as $\Phi^{-1}$ is continuous. 
 Further, $C^\infty (I,V)$ with the above topology is a Banach space, whence 
 $$h \colon C (I,V) \rightarrow C(I,S_{\LieA}),\quad f \mapsto \gamma_\Phi \circ (f + \Phi \circ u_1)$$ is continuous.
 Now as $\norm{\norm{q_1} f}_\infty \leq \norm{q_1}_\infty \norm{f}_\infty$ we deduce that $n_{q_1} \colon C(I,S_{\LieA}) \rightarrow C(I,\LieA \setminus \{0\}), f \mapsto \norm{q_1} f$ is continuous.
 In conclusion, $\theta = n_{q_1} \circ h$ is continuous.
\end{proof}
\end{appendix}

\bibliographystyle{abbrv}
\bibliography{SO3Curves}

\begin{thebibliography}{10}

\bibitem{bastiani64}
A.~Bastiani.
\newblock Applications diff\'erentiables et vari\'et\'es diff\'erentiables de
  dimension infinie.
\newblock {\em J. Analyse Math.}, 13:1--114, 1964.

\bibitem{bauer_new_2011}
M.~Bauer and M.~Bruveris.
\newblock A {New} {Riemannian} {Setting} for {Surface} {Registration}.
\newblock pages 182--193, Sept. 2011.

\bibitem{bauer_constructing_2014}
M.~Bauer, M.~Bruveris, S.~Marsland, and P.~W. Michor.
\newblock Constructing reparameterization invariant metrics on spaces of plane
  curves.
\newblock {\em Differential Geometry and its Applications}, 34:139--165, June
  2014.

\bibitem{bauer_overview_2014}
M.~Bauer, M.~Bruveris, and P.~W. Michor.
\newblock Overview of the {Geometries} of {Shape} {Spaces} and {Diffeomorphism}
  {Groups}.
\newblock {\em Journal of Mathematical Imaging and Vision}, pages 1--38, 2014.

\bibitem{MR3444354}
M.~Bauer, M.~Bruveris, and P.~W. Michor.
\newblock Why use {S}obolev metrics on the space of curves.
\newblock In {\em Riemannian computing in computer vision}, pages 233--255.
  Springer, Cham, 2016.

\bibitem{bauer_landmark-guided_2015}
M.~Bauer, M.~Eslitzbichler, and M.~Grasmair.
\newblock Landmark-{Guided} {Elastic} {Shape} {Analysis} of {Human} {Character}
  {Motions}.
\newblock {\em arXiv:1502.07666 [cs]}, Feb. 2015.

\bibitem{bauer_sobolev_2011}
M.~Bauer, P.~Harms, and P.~W. Michor.
\newblock Sobolev metrics on shape space of surfaces.
\newblock {\em Journal of Geometric Mechanics}, 3(4):389 -- 438, 2011.

\bibitem{carnegie-mellon_carnegie-mellon_2003}
{Carnegie-Mellon}.
\newblock Carnegie-{Mellon} {Mocap} {Database}., 2003.

\bibitem{celledoni_introduction_2014}
E.~Celledoni, H.~Marthinsen, and B.~Owren.
\newblock An {Introduction} to {Lie} {Group} {Integrators} - {Basics}, {New}
  {Developments} and {Applications}.
\newblock {\em J. Comput. Phys.}, 257:1040--1061, Jan. 2014.

\bibitem{celledoni_lie_1999}
E.~Celledoni and B.~Owren.
\newblock Lie {Group} {Methods} for {Rigid} {Body} {Dynamics} and {Time}
  {Integration} on {Manifolds}.
\newblock {\em Computer Methods in Applied Mechanics and Engineering},
  19:421--438, 1999.

\bibitem{MR0458335}
J.~Cheeger and D.~G. Ebin.
\newblock {\em Comparison theorems in {R}iemannian geometry}.
\newblock North-Holland Publishing Co., 1975.
\newblock North-Holland Mathematical Library, Vol. 9.

\bibitem{cotter_reparameterisation_2012}
C.~J. Cotter, A.~Clark, and J.~Peiró.
\newblock A {Reparameterisation} {Based} {Approach} to {Geodesic} {Constrained}
  {Solvers} for {Curve} {Matching}.
\newblock {\em International Journal of Computer Vision}, 99(1):103--121, Aug.
  2012.

\bibitem{MR1363804}
T.~Dobrowolski.
\newblock Every infinite-dimensional {H}ilbert space is real-analytically
  isomorphic with its unit sphere.
\newblock {\em J. Funct. Anal.}, 134(2):350--362, 1995.

\bibitem{MR1039321}
R.~Engelking.
\newblock {\em General topology}, volume~6 of {\em Sigma Series in Pure
  Mathematics}.
\newblock Heldermann Verlag, Berlin, second edition, 1989.

\bibitem{eslitzbichler_modelling_2014}
M.~Eslitzbichler.
\newblock Modelling character motions on infinite-dimensional manifolds.
\newblock {\em The Visual Computer}, pages 1--12, July 2014.

\bibitem{fuchs_shape_2009}
M.~Fuchs, B.~Jüttler, O.~Scherzer, and H.~Yang.
\newblock Shape {Metrics} {Based} on {Elastic} {Deformations}.
\newblock {\em Journal of Mathematical Imaging and Vision}, 35(1):86--102, May
  2009.

\bibitem{glockner_regularity_2012}
H.~Glöckner.
\newblock Regularity properties of infinite-dimensional {Lie} groups, and
  semiregularity, 2012.
\newblock {a}rXiv: 1208.0715 [math].

\bibitem{glockner15fos}
H.~Glöckner.
\newblock Fundamentals of submersions and immersions between
  infinite-dimensional manifolds, Mar. 2015.
\newblock {a}rXiv:1502.05795v3 [math].

\bibitem{gonzalez_castro_cyclic_2010}
G.~González~Castro, M.~Athanasopoulos, and H.~Ugail.
\newblock Cyclic animation using partial differential equations.
\newblock {\em The Visual Computer}, 26(5):325--338, 2010.

\bibitem{hausdorff_symbolische}
F.~{Hausdorff}.
\newblock {Die symbolische Exponentialformel in der Gruppentheorie}.
\newblock {Leipz. Ber. 58, 19-48.}, 1906.

\bibitem{hilgert12sag}
J.~Hilgert and K.~H. Neeb.
\newblock {\em Structure and geometry of {L}ie groups}.
\newblock Springer Monographs in Mathematics. Springer, New York, 2012.

\bibitem{iserles_lie-group_2000}
A.~Iserles, H.~Z. Munthe-Kaas, S.~P. Nørsett, and A.~Zanna.
\newblock Lie-group methods.
\newblock {\em Acta Numerica}, 9:215--365, Jan. 2000.

\bibitem{klassen_path-straightening_2005}
E.~Klassen and A.~Srivastava.
\newblock A path-straightening method for finding geodesics in shape spaces of
  closed curves in {R}3.
\newblock {\em SIAM Journal of Applied Mathematics}, 2005.

\bibitem{kovar_flexible_2003}
L.~Kovar and M.~Gleicher.
\newblock Flexible {Automatic} {Motion} {Blending} with {Registration}
  {Curves}.
\newblock In {\em Proceedings of the 2003 {ACM} {SIGGRAPH}/{Eurographics}
  {Symposium} on {Computer} {Animation}}, {SCA} '03, pages 214--224,
  Aire-la-Ville, Switzerland, Switzerland, 2003. Eurographics Association.

\bibitem{kovar_automated_2004}
L.~Kovar and M.~Gleicher.
\newblock Automated extraction and parameterization of motions in large data
  sets.
\newblock In {\em {ACM} {Transactions} on {Graphics} ({TOG})}, volume~23, pages
  559--568. ACM, 2004.

\bibitem{kovar_motion_2002}
L.~Kovar, M.~Gleicher, and F.~Pighin.
\newblock Motion {Graphs}.
\newblock {\em ACM Trans. Graph.}, 21(3):473--482, July 2002.

\bibitem{kriegl97tcs}
A.~Kriegl and P.~W. Michor.
\newblock {\em The convenient setting of global analysis}, volume~53 of {\em
  Mathematical Surveys and Monographs}.
\newblock American Mathematical Society, Providence, RI, 1997.

\bibitem{kriegl_regular_1997}
A.~Kriegl and P.~W. Michor.
\newblock Regular infinite dimensional {Lie} groups.
\newblock {\em Journal of Lie Theory}, 7:61--99, 1997.

\bibitem{kurtek_novel_2010}
S.~Kurtek, E.~Klassen, Z.~Ding, and A.~Srivastava.
\newblock A novel riemannian framework for shape analysis of 3d objects.
\newblock In {\em 2010 {IEEE} {Conference} on {Computer} {Vision} and {Pattern}
  {Recognition} ({CVPR})}, pages 1625 --1632, June 2010.

\bibitem{kurtek_elastic_2012}
S.~Kurtek and A.~Srivastava.
\newblock Elastic symmetry analysis of anatomical structures.
\newblock In {\em 2012 {IEEE} {Workshop} on {Mathematical} {Methods} in
  {Biomedical} {Image} {Analysis} ({MMBIA})}, pages 33 --38, Jan. 2012.

\bibitem{lahiri_precise_2015}
S.~Lahiri, D.~Robinson, and E.~Klassen.
\newblock Precise {Matching} of {PL} {Curves} in $\mathbb{R}^{N}$ in the
  {Square} {Root} {Velocity} {Framework}, 2015.
\newblock {a}rXiv:1501.00577 [math].

\bibitem{MR1666820}
S.~Lang.
\newblock {\em Fundamentals of differential geometry}, volume 191 of {\em
  Graduate Texts in Mathematics}.
\newblock Springer-Verlag, New York, 1999.

\bibitem{michor_overview_2007}
P.~W. Michor and D.~Mumford.
\newblock An overview of the {Riemannian} metrics on spaces of curves using the
  {Hamiltonian} approach.
\newblock {\em Applied and Computational Harmonic Analysis}, 23(1):74--113,
  July 2007.

\bibitem{mio_shape_2007}
W.~Mio, A.~Srivastava, and S.~Joshi.
\newblock On {Shape} of {Plane} {Elastic} {Curves}.
\newblock {\em Int. J. Comput. Vision}, 73(3):307--324, July 2007.

\bibitem{MR2261066}
K.-H. Neeb.
\newblock Towards a {L}ie theory of locally convex groups.
\newblock {\em Jpn. J. Math.}, 1(2):291--468, 2006.

\bibitem{pejsa_state_2010}
T.~Pejsa and I.~Pandzic.
\newblock State of the {Art} in {Example}-{Based} {Motion} {Synthesis} for
  {Virtual} {Characters} in {Interactive} {Applications}.
\newblock {\em Computer Graphics Forum}, 29(1):202--226, 2010.

\bibitem{MR3351079}
A.~Schmeding and C.~Wockel.
\newblock The {L}ie group of bisections of a {L}ie groupoid.
\newblock {\em Ann. Global Anal. Geom.}, 48(1):87--123, 2015.

\bibitem{sebastian_aligning_2003}
T.~Sebastian, P.~Klein, and B.~Kimia.
\newblock On aligning curves.
\newblock {\em IEEE Transactions on Pattern Analysis and Machine Intelligence},
  25(1):116--125, Jan. 2003.

\bibitem{sharon_2d-shape_2006}
E.~Sharon and D.~Mumford.
\newblock 2d-{Shape} {Analysis} {Using} {Conformal} {Mapping}.
\newblock {\em International Journal of Computer Vision}, 70(1):55--75, June
  2006.

\bibitem{shoemake_animating_1985}
K.~Shoemake.
\newblock Animating {Rotation} with {Quaternion} {Curves}.
\newblock {\em SIGGRAPH Comput. Graph.}, 19(3):245--254, July 1985.

\bibitem{srivastava_statistical_2005}
A.~Srivastava, S.~Joshi, W.~Mio, and X.~Liu.
\newblock Statistical {Shape} {Analysis}: {Clustering}, {Learning}, and
  {Testing}.
\newblock {\em IEEE Trans. Pattern Anal. Mach. Intell}, 27:590--602, 2005.

\bibitem{srivastava_shape_2011}
A.~Srivastava, E.~Klassen, S.~Joshi, and I.~Jermyn.
\newblock Shape {Analysis} of {Elastic} {Curves} in {Euclidean} {Spaces}.
\newblock {\em IEEE Transactions on Pattern Analysis and Machine Intelligence},
  33(7):1415 --1428, July 2011.

\bibitem{su_statistical_2014}
J.~Su, S.~Kurtek, E.~Klassen, and A.~Srivastava.
\newblock Statistical analysis of trajectories on {Riemannian} manifolds:
  {Bird} migration, hurricane tracking and video surveillance.
\newblock {\em The Annals of Applied Statistics}, 8(1):530--552, Mar. 2014.

\bibitem{su_rate-invariant_2014}
J.~Su, A.~Srivastava, F.~de~Souza, and S.~Sarkar.
\newblock Rate-{Invariant} {Analysis} of {Trajectories} on {Riemannian}
  {Manifolds} with {Application} in {Visual} {Speech} {Recognition}.
\newblock In {\em 2014 {IEEE} {Conference} on {Computer} {Vision} and {Pattern}
  {Recognition} ({CVPR})}, pages 620--627, June 2014.

\bibitem{younes_computable_1998}
L.~Younes.
\newblock Computable {Elastic} {Distances} between {Shapes}.
\newblock {\em SIAM Journal on Applied Mathematics}, 58(2):565--586, Apr. 1998.

\bibitem{younes_spaces_2012}
L.~Younes.
\newblock Spaces and manifolds of shapes in computer vision: {An} overview.
\newblock {\em Image and Vision Computing}, 30(6–7):389--397, June 2012.

\end{thebibliography}
\end{document}